\documentclass[12pt]{amsart}
\usepackage{amsfonts,amsmath,amssymb,amsthm,latexsym}
\usepackage{amscd,esint}
\usepackage[plainpages=false]{hyperref}
\usepackage{graphicx}
\usepackage{float}
\usepackage{subfigure}
\usepackage{xcolor}
\usepackage{enumitem}

\numberwithin{equation}{section}

\newcommand{\beq}{\begin{equation}}
\newcommand{\eeq}{\end{equation}}
\newcommand{\beqs}{\begin{eqnarray*}}
\newcommand{\eeqs}{\end{eqnarray*}}
\newcommand{\beqn}{\begin{eqnarray}}
\newcommand{\eeqn}{\end{eqnarray}}
\newcommand{\beqa}{\begin{array}}
\newcommand{\eeqa}{\end{array}}

\newcommand{\cH}{\mathcal{H}}

\newcommand{\cO}{\mathcal{O}}

\newcommand{\bS}{\mathbb{S}}
\newcommand{\bC}{\mathbb{C}}
\newcommand{\bR}{\mathbb{R}}

\newcommand{\bD}{\mathbb{D}}

\newcommand{\sq}{\sqrt{-1}}

\newcommand{\D}{\nabla}
\newcommand{\p}{\partial}
\newcommand{\ol}{\overline}

\newcommand{\lp}{\left(}
\newcommand{\rp}{\right)}

\newcommand{\eps}{\varepsilon}

\newcommand{\Om}{\Omega}

\DeclareMathOperator{\re}{Re}
\DeclareMathOperator{\im}{Im}

\newtheorem{prop}{Proposition}[section]
\newtheorem{theo}[prop]{Theorem}
\newtheorem{lem}[prop]{Lemma}

\newtheorem{cor}[prop]{Corollary}
\newtheorem{rem}[prop]{Remark}

\newtheorem{defi}[prop]{Definition}

\allowdisplaybreaks
\hypersetup{
  hidelinks,
  pdftitle={Singular solutions to the complex Monge--Ampere equation II}
}

\title{Singular solutions to the complex Monge--Amp\`ere equation II}
\author[J. Wang]{Jiaxiang Wang$^{*}$}
\address{School of Mathematical Sciences and LPMC,
Nankai University, Tianjin 300071,
People's Republic of China}
\email{wangjx\_manifold@126.com}

\author[W. Wang]{Wenlong Wang$^{\dagger}$}
\address{School of Mathematical Sciences and LPMC,
Nankai University, Tianjin 300071,
People's Republic of China}
\email{wangwl@nankai.edu.cn}

\thanks{*Jiaxiang Wang was partially supported by the National Natural Science Foundation of China
(Grant Nos.~12401068 and 12571060).}
\thanks{$\dagger$Wenlong Wang was partially supported by the National Natural Science Foundation of China (Grant Nos. 12471054 and 12001292) and the Fundamental Research Funds for the Central Universities, Nankai University (Grant No. 050-63263075). }

\begin{document}

\begin{abstract}
Continuing our work in \cite{WangWang25}, we construct new singular solutions to the complex Monge--Ampère equation with constant right-hand side in the Bedford--Taylor sense. Let \(\mathbb D\) denote the open unit disk in \(\mathbb C\). For \(n\geq2\), these solutions include:
\begin{enumerate}[
    label=(\roman*),
    widest=iii,
    align=right,
    labelindent=0pt,
    labelsep=0.3em,
    itemindent=0pt,
    leftmargin=*,
    nosep
]
    \item a bounded solution on \(\mathbb D^n\) that is
    discontinuous almost everywhere and nowhere classically
    differentiable;

    \item a continuous exhaustion of \(\mathbb C^n\) that
    does not belong to \(W^{2,1}_{\mathrm{loc}}(\mathbb C^n)\);

    \item a continuous solution on \(\overline{\mathbb D^n}\)
    whose Laplacian has a singular part.
\end{enumerate}
We construct these solutions from singular harmonic functions using partial Legendre transforms and holomorphic pullbacks.
\end{abstract}

\maketitle

\tableofcontents

\section{Introduction}\label{sec:introduction}
The complex Monge--Amp\`ere operator plays an important role in complex analysis and geometry. There are two fundamental approaches to defining weak solutions of the complex Monge--Amp\`ere equation. One is the pluripotential approach developed by Bedford and Taylor \cite{BT76}. The other is the viscosity approach, introduced by Crandall and Lions for nonlinear PDEs \cite{CL} and subsequently developed for the complex Monge--Ampère equation in \cite{EGZ,GLZ,HL09,HL13,WangY}.

Understanding the regularity of such generalized solutions has been a major focus of research on the complex Monge--Amp\`ere equation. A complementary direction is to construct singular solutions and investigate their singularity profiles, thereby clarifying the types of singularities that can occur and the limits of regularity. 

We study the complex Monge--Ampère equation in its simplest form,
\begin{equation}\label{CMAdensity}
\det(u_{i\bar j})=1,
\end{equation}
for locally bounded plurisubharmonic functions \(u\) on domains in \(\mathbb C^n\). In the Bedford--Taylor sense (see Section \ref{pre} for details), this equation is understood as
\begin{equation}\label{CMApluipotential}
\left(dd^cu\right)^n=\left(dd^c|z|^2\right)^n,
\end{equation}
where \(d^c=i(\bar\partial-\partial)\).

In \cite{WangWang25}, the authors constructed the following singular solution on \(\mathbb D\times\mathbb C\) of \eqref{CMAdensity} in both the Bedford--Taylor and
viscosity senses, namely
\begin{equation}\label{notdini}
u(z_1,z_2)=-\frac{\left|\re z_2\right|^2}{\log |z_1|}-2|z_1|^2\left(\log |z_1|-1\right).
\end{equation}
This solution is continuous but not Dini continuous; it belongs to \(W^{2,1}_{\mathrm{loc}}\) but not to \(W^{2,p}_{\mathrm{loc}}\) for any \(p>1\). In this paper, we continue this study in two directions: whether even \emph{more singular} solutions exist and whether \emph{entireness} and \emph{exhaustion} rule out highly singular solutions.

Throughout the paper,
for simplicity, we use $\beta_n$ to denote $\left(dd^c|z|^2\right)^n$; 
for a domain $\Omega\subset \bC^n$, denote by $PSH(\Omega)$ the plurisubharmonic functions on $\Omega$. 

In \cite{WangWang25}, we asked whether solutions of \eqref{CMAdensity} necessarily belong to \(W^{2,1}_{\mathrm{loc}}\). Extending the construction of \eqref{notdini}, we first obtain a continuous, non-proper counterexample.

\begin{theo}\label{theo:nonproper-w21-main}
There exists a non-proper function
\(u_1\in PSH(\mathbb D\times\mathbb C)\cap C(\overline{\mathbb D}\times\mathbb C)\)
that solves \eqref{CMAdensity} in both the Bedford--Taylor and viscosity senses. This solution is not Dini continuous and does not belong to \(W^{2,1}_{\mathrm{loc}}(\mathbb D\times\mathbb C)\).
\end{theo}
This contrasts sharply with the real Monge--Ampère equation, for which convex Alexandrov solutions with density bounded above and below by positive constants belong to \(W^{2,1}_{\mathrm{loc}}\); see \cite{DePhilippisFigalli13,Mooney15Partial,Mooney15W21}.

The same approach also yields a bounded solution that is discontinuous almost everywhere and nowhere classically differentiable.

\begin{theo}\label{theo:dense-poles-main}
There exists a function $u_2\in PSH(\bD^2)\cap L^\infty(\bD^2)$
that solves \eqref{CMAdensity} in the Bedford--Taylor sense.
This solution is discontinuous almost everywhere and nowhere
classically differentiable. Moreover, \(u_2\notin W^{2,1}_{\mathrm{loc}}(\mathbb D^2)\).
\end{theo}

We next show that the failure of $W^{2,1}_{\mathrm{loc}}$ regularity
also occurs for continuous exhaustion solutions on $\bC^2$.

\begin{theo}\label{theo:proper-w21-main}
There exists a continuous exhaustion \(u_3\) of \(\bC^2\)
that solves \eqref{CMAdensity} in both the Bedford--Taylor
and viscosity senses and is invariant under the anti-diagonal
\(\mathbb S^1\)-action
\((z_1,z_2)\mapsto(e^{i\theta}z_1,e^{-i\theta}z_2)\).
This solution is not Dini continuous and belongs neither to
\(W^{2,1}_{\mathrm{loc}}(\bC^2)\) nor to
\(W^{1,p}_{\mathrm{loc}}(\bC^2)\) for any \(p>2\).
\end{theo}

Subtracting a constant from \(u_3\) and restricting it to a connected component of a sublevel set yields the following corollary.

\begin{cor}\label{cor:bounded-sublevel}
There exists a bounded hyperconvex domain \(\Omega\subset\mathbb C^2\) and a function \(v\in PSH(\Omega)\cap C(\overline{\Omega})\) that solves \eqref{CMAdensity} in both the Bedford--Taylor and viscosity senses, with zero boundary values. This solution is not Dini continuous and does not belong to \(W^{2,1}_{\mathrm{loc}}(\Omega)\). Both \(\Omega\) and \(v\) are invariant under the above \(\mathbb S^1\)-action.
\end{cor}

B{\l}ocki asked whether the continuous solution
of the complex Monge--Amp\`ere equation on a bounded hyperconvex
domain, with zero boundary values and a smooth density bounded
away from zero, must be smooth in the interior; see
\cite[p.~451]{Blocki14CRMA} and \cite[Question~23]{DGZ16}.
Corollary~\ref{cor:bounded-sublevel} provides a highly singular counterexample, even with constant density and anti-diagonal \(\mathbb S^1\)-symmetry.

The non-$W^{2,1}_{\mathrm{loc}}$ solutions above nevertheless have locally integrable distributional Laplacians. The following theorem shows that even this weaker regularity can fail.

\begin{theo}\label{theo:singular-laplacian-main}
There exists a function
$u_4\in PSH(\bD^2)\cap C(\overline{\bD^2})$ that solves
\eqref{CMAdensity} in both the Bedford--Taylor and viscosity senses.
The distributional Laplacian of $u_4$ has a nonzero singular part.
\end{theo}

The higher-dimensional analogues of the above theorems and corollary follow by adding \(|z_3|^2+\cdots+|z_n|^2\) to the two-dimensional potentials.

We recall several known examples that illustrate typical forms of singular behavior. Błocki \cite{{Blo1}} constructed the solution
$$
v(z)=\frac{n}{n-1}\left(1+|z_1|^2\right)|z'|^{2-\frac{2}{n}},
$$
where \(z'=(z_2,\ldots,z_n)\). This function solves
$
\det(v_{i\bar j})=\left(1+|z_1|^2\right)^{n-2}
$
in both the pluripotential and viscosity senses for \(n\geq2\). Błocki and Dinew \cite{BlockiDinew11} proved that a plurisubharmonic solution of the complex Monge--Ampère equation with a smooth, strictly positive right-hand side is locally smooth if \(u\in W^{2,p}_{\mathrm{loc}}\) for some \(p>n(n-1)\). The example above shows that the threshold \(n(n-1)\) cannot be lowered. 

He \cite{He12} constructed the solution
\(
w(z)=n^{\frac{2}{n}}|z_1|^{\frac{2}{n}}\left(1+|z'|^2\right),
\)
which solves \eqref{CMAdensity} in both the pluripotential and viscosity senses. For \(n\geq3\), this solution is not locally Lipschitz, showing that local boundedness alone does not imply an interior gradient bound. In \cite{WangWang25}, the authors constructed two families of entire toric solutions whose optimal Hölder exponents range over \((0,1)\) and whose \(W^{2,p}_{\mathrm{loc}}\) integrability thresholds range over \((1,2)\). These families rule out a universal positive interior Hölder exponent even under toric symmetry and entireness.

These examples show that a smooth positive right-hand side alone does not guarantee interior regularity. For the Dirichlet problem, regularity can be obtained under suitable assumptions on the domain and the boundary data. We recall some classical results in this direction.

For a bounded strongly pseudoconvex domain \(\Omega\subset\mathbb C^n\), Bedford and Taylor \cite{BT76} established the existence and uniqueness of continuous pluripotential solutions to the Dirichlet problem with a nonnegative density in \(C(\overline\Omega)\) and continuous boundary data. On a smoothly bounded strongly pseudoconvex domain, Caffarelli, Kohn, Nirenberg, and Spruck \cite{CKNS85} proved that the Dirichlet problem admits a unique strictly plurisubharmonic solution in \(C^\infty(\overline\Omega)\) for every smooth boundary datum and every \(f\in C^\infty(\overline\Omega)\) with \(f>0\) on \(\overline\Omega\).

In the pluripotential setting, continuity results for the Dirichlet problem involve assumptions on both the domain and the boundary data. Kołodziej \cite{K98} proved continuity up to the boundary for nonnegative \(L^p\)-densities (\(p>1\)), and continuous boundary data on bounded strictly pseudoconvex domains with smooth boundary. The role of the boundary data is explicit in Błocki’s theorem \cite{Blocki96Hyperconvex}. Let \(\Omega\) be a bounded hyperconvex domain. If \(F\in C(\overline\Omega)\), \(F\geq0\), and \(\varphi\in C(\partial\Omega)\) is the trace of some \(v_0\in PSH(\Omega)\cap C(\overline\Omega)\), then the Dirichlet problem with density \(F\) and boundary value \(\varphi\) admits a unique solution in \(PSH(\Omega)\cap C(\overline\Omega)\).

Guedj, Kołodziej, and Zeriahi \cite{GKZ08} studied Hölder regularity for the Dirichlet problem on a smoothly bounded strongly pseudoconvex domain \(\Omega\subset\mathbb C^n\), with a nonnegative density \(f\in L^p(\Omega)\), \(p>1\), and Hölder continuous boundary data. They proved that the solution is Hölder continuous on \(\overline\Omega\), provided that the solution of the corresponding homogeneous Dirichlet problem with the same boundary data has gradient in \(L^2(\Omega)\).

The solution in \eqref{notdini} and those in the above theorems are constructed using the same method: the \emph{partial Legendre transformation}, which takes the Legendre transform with respect to one variable while keeping the others fixed. In real dimension two, its use to linearize the Monge--Ampère equation with constant right-hand side goes back to Darboux. In complex dimension two, we use this transformation to construct smooth solutions of \eqref{CMAdensity} from harmonic functions in three real variables that are strictly convex in one variable. The resulting solutions are independent of the imaginary part of one complex coordinate.

More precisely, let \(\Omega\subset\mathbb C\times\mathbb R\) be open, with coordinates \((z_1,t)\), and assume that every nonempty fiber
\(
\Omega_{z_1}=\{t\in\mathbb R\mid (z_1,t)\in\Omega\}
\)
is connected. Let \(v\in C^2(\Omega)\) satisfy \(v_{tt}>0\), and define
\[
\Phi_v:\Omega\longrightarrow\mathbb C\times\mathbb R,
\qquad
\Phi_v(z_1,t)=\bigl(z_1,v_t(z_1,t)\bigr).
\]
Then \(\Phi_v\) is a \(C^1\) diffeomorphism onto its open image \(\Phi_v(\Omega)\). Write
\[
\Phi_v^{-1}(z_1,s)
=\bigl(z_1,t_v(z_1,s)\bigr).
\]
Set
\[
F=
\left\{
(z_1,z_2)\in\mathbb C^2
\;\middle|\;
\bigl(z_1,\re z_2\bigr)\in\Phi_v(\Omega)
\right\}.
\]
We define the normalized partial Legendre transform of \(v\) on \(F\) by
\[
u(z_1,z_2)
=4\bigl(s\,t_v(z_1,s)-v(z_1,t_v(z_1,s))\bigr),
\qquad s=\re z_2.
\]

To obtain singular solutions, we combine the above construction with the following criterion for extending solutions across locally complete pluripolar sets. Recall that a subset \(A\) of a complex manifold is called locally complete pluripolar if, near each point of \(A\), it coincides with the \(-\infty\)-locus of a plurisubharmonic function that is not identically \(-\infty\). See Section \ref{pre} for further details.

\begin{theo}\label{main1}
Let \(\Omega\) and \(v\) be as above, and assume in addition that \(v\) is harmonic. Then $u$
is a smooth plurisubharmonic solution of \eqref{CMAdensity} on $F$ that is independent of $\im z_2$.
Set 
\[
E=\operatorname{Int}\lp \ol{F}\rp, \qquad S=E\setminus F,
\]
where \(\overline{F}\) denotes the closure of \(F\) in \(\mathbb C^2\). If $S$ is locally complete pluripolar in
$E$ and $u$ is locally bounded near $S$, then $u$ extends uniquely to a function in \(PSH(E)\cap L^\infty_{\mathrm{loc}}(E)\). 
This extension remains independent of $\im z_2$ and satisfies \eqref{CMAdensity} on $E$ in the Bedford--Taylor sense.
\end{theo}
For example, the solution \eqref{notdini} is obtained by extending the normalized partial Legendre transform of the harmonic function
\[
v(z_1,t)
=-t^2\log|z_1|
+\frac12|z_1|^2\bigl(\log|z_1|-1\bigr)
\]
on \((\mathbb D\setminus\{0\})\times\mathbb R\). Here, \(v_{tt}=-2\log|z_1|>0\) is a constant multiple of the fundamental solution of the planar Laplacian. The logarithmic singularity of \(v_{tt}\) leads to the reciprocal-logarithmic behavior of the solution.

The superposition principle for harmonic functions allows us to construct further solutions of \eqref{CMAdensity}. 

We use weighted sums of logarithmic potentials to construct \(u_1\) and \(u_2\). For \(u_1\), we choose a convergent sequence of poles and include its limit point as an additional pole. Continuity at each isolated pole follows as in the single-pole case, while the additional pole ensures continuity at the limit point. The pole locations and weights are chosen so that disjoint neighborhoods of the poles make nonsummable contributions to the \(L^1\) norm of the real Hessian near the limit point.

For \(u_2\), we choose a countable dense set of poles. The reciprocal of the superposed logarithmic potential vanishes at these poles but is positive almost everywhere, so their density forces the resulting solution to be discontinuous almost everywhere.

To construct \(u_3\), we start with the harmonic dual of Warren’s non-quadratic solution to \eqref{CMAdensity} on \(\mathbb C^2\) \cite{Warren16}. After a simple linear change of variables, this solution takes the form
\begin{equation}\label{War}
U_W(z_1,z_2)=e^{2\operatorname{Re}z_2}+|z_1|^2e^{-2\operatorname{Re}z_2}
\end{equation}
Inverting the normalized partial Legendre transform gives its harmonic dual
\begin{equation*}
h_W(z_1,t)=\frac12\lp t\log\left(t+\sqrt{|z_1|^2+t^2}\right)-\sqrt{|z_1|^2+t^2}\rp.
\end{equation*}

A key property of $h_W$ is that 
$$(h_W)_{tt}=\frac{1}{2\sqrt{|z_1|^2+t^2}},$$ 
a positive multiple of the fundamental solution of \(\Delta\) in \(\mathbb R^3\), which provides the singular building block needed for our construction. We form a harmonic function \(v_3\) by adding to \(h_W\) suitably chosen positive weighted sums of its translates in \(z_1\) and \(t\) and reflections in \(t\). The normalized partial Legendre transform then extends to a continuous solution that fails to belong to \(W^{2,1}_{\mathrm{loc}}\). However, its independence of \(\operatorname{Im}z_2\) prevents it from being proper on \(\mathbb C^2\).

To obtain an exhaustion, we use the locally defined holomorphic map
\begin{equation}\label{pullback}
\Pi(z_1,z_2)=\bigl(z_1z_2,\log z_2\bigr)
\end{equation}
on \(\mathbb C\times\mathbb C^*\), which transforms the Warren solution into the standard solution:
\[
(U_W\circ\Pi)(z_1,z_2)=|z_1|^2+|z_2|^2.
\]
For solutions independent of \(\im z_2\), the pullback is independent of the logarithm branch and invariant under the anti-diagonal \(\mathbb S^1\)-action. 
Moreover, \(\det D_{\mathbb C}\Pi=1\), so the pullback preserves the equation. This leads to the following theorem.

\begin{theo}\label{main:legendre-reductions}
Let \(\Omega\) and \(v\) be as in Theorem \ref{main1}, and set
\[
\hat F=\left\{(z_1,z_2)\in\bC\times\bC^*\,\big|\,\left(z_1z_2,\log |z_2|\right)
\in\Phi_v(\Om)\right\}.
\]
Let \(v^*\) denote the partial Legendre transform of \(v\). Then the function
\[
\hat u(z_1,z_2)=4v^*\left(z_1z_2,\log |z_2|\right)
\]
is a smooth, plurisubharmonic anti-diagonal $\bS^1$-invariant solution of
\eqref{CMAdensity} on $\hat F$. Set
\[
\hat E=\operatorname{Int}\bigl(\overline{\hat F}\bigr),
\qquad \hat S=\hat E\setminus\hat F.
\]
If \(\hat S\) is locally complete pluripolar in \(\hat E\) and \(\hat u\) is locally bounded near \(\hat S\), then \(\hat u\) extends uniquely to a function in \(PSH(\hat E)\cap L^\infty_{\mathrm{loc}}(\hat E)\). The extension remains invariant under the same action and satisfies the equation on \(\hat E\) in the Bedford–Taylor sense.
\end{theo}

To construct \(u_4\), we first take \(g=g(z_1)\) to be the logarithmic potential of a Cantor measure. This gives a continuous subharmonic function whose distributional Laplacian is a singular measure. We then use the partial Legendre transformation to construct a plurisubharmonic function \(W\) such that \(u_4=W+g\) has complex Monge–Ampère measure \(\beta_2\). Since \(\Delta u_4\geq\Delta g\), the distributional Laplacian of \(u_4\) has a singular part.

In Section~\ref{sec:further-examples}, we construct further solutions using partial Legendre transforms and volume-preserving holomorphic pullbacks. Starting from Warren’s harmonic dual, we interpolate between it and its reflection in \(t\) to obtain a family of entire solutions with Hölder exponents ranging over \((0,1]\). We also explicitly compute the transform of the sum of two translates. By adding a harmonic quadratic polynomial to the dual and applying the transform followed by the holomorphic pullback $\Pi$, we recover LeBrun’s Taub–NUT potential. We also derive an explicit solution by applying the partial Legendre transform to a cubic harmonic polynomial. Its real counterpart provides an explicit counterexample to global \(C^2\) regularity for \(C^{2,1}\) boundary data on a disk. Finally, using holomorphic pullbacks, we construct a family of continuous entire toric solutions parametrized by invertible row-stochastic matrices, unifying and extending the families obtained in our earlier work \cite{WangWang25}.

The partial Legendre transform has been used to obtain a wide range
of results. For instance, see \cite{GP} for the preservation of
ellipticity; \cite{Gua,RSW1,SW1,RSW2,DS,Liu} for applications to
regularity theory for Monge--Amp\`ere equations; \cite{Wa,WZ1}
for estimates for linearized Monge--Amp\`ere equations and
fourth-order equations; \cite{Ki} for Kiselman's minimum principle
in pluripotential theory; and \cite{Donaldson10,He12Donaldson}
for harmonic formulations of Donaldson's equation using the
transform and its variants.

For other constructions of explicit solutions to the complex Monge–Ampère equation and related Hessian equations using different ansatzes, see Pan and Zhang \cite{PanZhang26} and Tsai, Tsui, and Wang \cite{TTW}.

\smallskip
\noindent{\it Organization.}
Section~\ref{pre} collects preliminaries used in our constructions,
including Bedford--Taylor and viscosity solutions, extension across
pluripolar sets, and Cantor measures.
Section~\ref{partial Legendre} develops the partial Legendre transform
for solutions independent of $\operatorname{Im}z_2$ and for solutions
invariant under the anti-diagonal $\bS^1$-action.
Section~\ref{sec:logarithmic-poles} uses infinite superpositions of
logarithmic poles to construct a continuous solution outside
$W^{2,1}_{\mathrm{loc}}$ and a bounded solution that is discontinuous
almost everywhere and nowhere classically differentiable.
Section~\ref{sec:proper-w21} constructs a continuous exhaustion
of $\bC^2$ that is not Dini continuous and belongs neither to
$W^{2,1}_{\mathrm{loc}}$ nor to $W^{1,p}_{\mathrm{loc}}$
for any $p>2$.
Section~\ref{sec:cantor-laplacian} constructs a continuous solution
whose distributional Laplacian has a singular part.
Finally, Section~\ref{sec:further-examples} presents further explicit
solutions obtained through partial Legendre transforms and
volume-preserving holomorphic pullbacks.

\subsection*{Acknowledgement}
The authors thank Bin Zhou for asking whether bounded solutions of \(\det(\partial\bar\partial u)=1\) are necessarily continuous, and Jiakun Liu for asking how properness affects the regularity of solutions. These questions helped motivate the present work. The authors thank Yu Yuan for introducing them to the method of constructing solutions via volume-preserving holomorphic pullbacks. This led them to observe that Warren’s solution pulls back to the standard quadratic potential under the map \(\Pi\) defined in \eqref{pullback}. This observation further inspired them to combine the pullback with superpositions of translates of Warren’s harmonic dual, leading to the construction of the exhaustion solution outside \(W^{2,1}_{\mathrm{loc}}\). The authors also thank their colleague Xiao Yao for a conversation that inspired them to consider singular measures supported on self-similar sets arising in complex dynamics, although they ultimately chose the simpler and more familiar Cantor set. Finally, they thank Xu-Jia Wang for his encouragement.

\subsection*{Development of the results and use of AI}
The authors began this project in 2024, building on their work in the first paper in this series \cite{WangWang25}. In that paper, they had already used the partial Legendre transform to construct the solution \eqref{notdini}, although this approach was not explicitly described there. The authors subsequently extended the construction of \eqref{notdini} to infinitely many poles, obtaining the solutions \(u_1\) and \(u_2\). The initial ideas for \(u_3\) and \(u_4\) were inspired by the authors’ conversations with other mathematicians, as noted in the acknowledgements above. The authors subsequently developed these ideas into concrete construction strategies.

The mathematical ideas and constructions were developed by the authors. AI tools were used for language editing and assistance with literature searches. The authors take full responsibility for the content of the manuscript.

\section{Preliminaries}\label{pre}

We first recall the Bedford--Taylor construction \cite{BT76}.
On $\bC^n$, we use the convention $d=\partial+\bar\partial$
and $d^c=\sq(\bar\partial-\partial)$, so that
$dd^c=2\sq\partial\bar\partial$.
For a domain $\Om\subset\bC^n$, denote by $PSH(\Om)$ the set of
plurisubharmonic functions on $\Om$.
For $u\in PSH(\Om)$, the current $dd^cu$ is closed and positive.
If $u$ is also locally bounded, its complex Monge--Amp\`ere
measure is defined inductively by
$$
\lp dd^cu\rp^k
:=dd^c\big(u\lp dd^cu\rp^{k-1}\big),
\qquad k=1,\ldots,n,
$$
with $\lp dd^cu\rp^0=1$.

If $u\in C^2(\Om)$, then
$$
dd^cu=2\sq\sum_{i,j=1}^n
u_{i\bar j}\,dz_i\wedge d\bar z_j,
$$
and hence
$$
\lp dd^cu\rp^n=\det(u_{i\bar j})\,\beta_n,
$$
where $\beta_n=\lp dd^c|z|^2\rp^n$, as in the introduction.

In this paper, we will use the following convergence theorem to verify the
pluripotential solutions below.

\begin{lem}[Convergence theorem \cite{BT82}]\label{convergence-theorem}
Let $u_j\in PSH(\Om)\cap
L^{\infty}_{\mathrm{loc}}(\Om)$ decrease to
$u\in PSH(\Om)\cap L^{\infty}_{\mathrm{loc}}(\Om)$.
Then $\lp dd^c u_j\rp^n$ converges weakly to $\lp dd^c u\rp^n$ on $\Om$.
\end{lem}
The same conclusion holds if \(u_j\to u\) locally uniformly on \(\Omega\).

As mentioned in the Introduction, we will extend solutions across
certain singular sets. We first recall the relevant notions of
\emph{pluripolarity}.

\begin{defi}
Let $\Om\subset\bC^n$ be open and let $S\subset\Om$.
The set $S$ is locally pluripolar in $\Om$ if, for every $z\in S$,
there exist a connected open neighborhood $U\subset\Om$ of $z$
and a function $u\in PSH(U)$, not identically $-\infty$, such that
$$
S\cap U\subset\{w\in U\,|\,u(w)=-\infty\}.
$$
If, for every $z\in S$, the neighborhood $U$ and the function $u$
can be chosen so that equality holds, then $S$ is called
locally complete pluripolar in $\Om$.
\end{defi}

Every locally pluripolar subset of an open set in $\bC^n$ has
Lebesgue measure zero \cite{BT82}.
A complex-analytic subset of a complex manifold is locally
complete pluripolar if it contains no connected component
of the manifold.

We will use the following extension lemma for locally bounded
plurisubharmonic functions.

\begin{lem}\label{lem:pluripolar-extension}
Let $\cO\subset\bC^n$ be open, and let $A\subset\cO$ be a closed
locally complete pluripolar set. Suppose that
\(
v\in PSH(\cO\setminus A)
\cap L^\infty_{\mathrm{loc}}(\cO\setminus A)
\)
is locally bounded near $A$.
Then $v$ has a unique plurisubharmonic extension
$$
\tilde v\in PSH(\cO)\cap L^\infty_{\mathrm{loc}}(\cO),
$$
given by
\begin{equation}\label{eq:canonical-psh-extension}
\tilde v=v\quad\text{on }\cO\setminus A,
\qquad
\tilde v(\zeta)
=\limsup_{\cO\setminus A\ni z\to\zeta}v(z)
\quad\text{for }\zeta\in A.
\end{equation}
Any continuous extension of $v$ agrees with $\tilde v$.
Moreover, if $\lp dd^cv\rp^n=f\,\beta_n$ on $\cO\setminus A$
in the Bedford--Taylor sense for some
$f\in L^1_{\mathrm{loc}}(\cO)$, then
$\lp dd^c\tilde v\rp^n=f\,\beta_n$ on $\cO$
in the same sense.
\end{lem}

\begin{proof}
The removable-singularity theorem for plurisubharmonic functions
\cite[Chapter~I, Theorem~5.24]{DemaillyCADG} gives the unique
plurisubharmonic extension defined by
\eqref{eq:canonical-psh-extension}.
This formula and the local boundedness of $v$ near $A$ imply
that $\tilde v$ is locally bounded on $\cO$.
The same formula shows that any continuous extension of $v$
agrees with $\tilde v$.

Since $\tilde v$ is locally bounded, its Bedford--Taylor
measure does not charge the locally pluripolar set $A$
\cite[Theorem~6.9]{BT82}.
The measure $f\,\beta_n$ also vanishes on $A$, since $A$ has
Lebesgue measure zero.
By locality, the two measures agree on $\cO\setminus A$.
They therefore agree on all of $\cO$.
\end{proof}

Next, we recall the notion of viscosity solutions.

\begin{defi}[\cite{CIL,EGZ}]\label{vissol}
Let $\Om\subset\bC^n$ be a domain, let $f\in C(\Om)$ be
nonnegative, and let $\cH_n$ be the space of $n\times n$ Hermitian
matrices. Define $F:\Om\times\cH_n\to\bR\cup\{+\infty\}$ by
$$
F(x,A) = \begin{cases}
f(x)-\det(A), & \text{if } A\geq0, \\
+\infty, & \text{otherwise}.
\end{cases}
$$
An upper semicontinuous function $u:\Om\to\bR$ is a viscosity
subsolution of $\det(\partial\bar\partial u)=f$ if, whenever
$p\in\Om$ and $\varphi\in C^2$ near $p$ are such that
$u-\varphi$ has a local maximum at $p$, we have
$$
F\bigl(p,\partial\bar\partial\varphi(p)\bigr)\leq0.
$$
Similarly, a lower semicontinuous function $u:\Om\to\bR$ is a viscosity
supersolution if the reverse inequality holds whenever $u-\varphi$ has a local minimum. A function $u$ is a viscosity solution if it is both a viscosity subsolution and a viscosity supersolution.
\end{defi}

Eyssidieux, Guedj, and Zeriahi \cite{EGZ} proved the equivalence between viscosity and pluripotential subsolutions of \(\det(\partial\bar\partial u)=f\). We will use the fact that every continuous solution in Bedford–Taylor sense of \eqref{CMAdensity} is also a viscosity solution. Although this fact is well known, we include a proof for the reader’s convenience.

\begin{prop}\label{prop:bt-viscosity}
Let $\Om\subset\bC^n$ be a domain and let
$u\in PSH(\Om)\cap C(\Om)$. Suppose that $u$ satisfies
$$\lp dd^cu\rp^n=\beta_n$$
in the Bedford--Taylor sense. Then \(u\) is a viscosity solution of the same equation.
\end{prop}

\begin{proof}
By \cite[Proposition 1.5]{EGZ}, \(u\) is a viscosity subsolution. It therefore suffices to prove that \(u\) is a viscosity supersolution. Let \(p\in\Omega\) and let \(\varphi\) be a \(C^2\) function touching \(u\) from below at \(p\). Assume that \((\varphi_{i\bar j}(p))\ge0\). We need to show that
\[
\det(\varphi_{i\bar j}(p))\le1.
\]
Suppose, to the contrary, that \(\det(\varphi_{i\bar j}(p))>1\). Then \((\varphi_{i\bar j}(p))>0\). By continuity, for sufficiently small \(\varepsilon,R>0\), the function \(\varphi-\varepsilon|z-p|^2\) is strictly plurisubharmonic with complex Hessian determinant greater than \(1\) on \(B_R(p)\). We may also assume that \(\overline{B_R(p)}\subset\Omega\) and \(\varphi\le u\) there. Set
\[
\psi(z)=\varphi(z)-\varepsilon|z-p|^2+\frac{\varepsilon R^2}{2}.
\]
Then \(\psi<u\) on \(\partial B_R(p)\), while \(\psi(p)>u(p)\). Since
\[
(dd^c\psi)^n>\beta_n=(dd^cu)^n,
\]
the Bedford–Taylor comparison principle \cite{BT82} gives \(\psi\le u\) on \(B_R(p)\), a contradiction. Thus $u$ is also a viscosity supersolution.
\end{proof}

We next recall two facts that will be used to rule out
\(W^{2,1}_{\mathrm{loc}}\) regularity for certain singular solutions.
For completeness, we include their proofs.

\begin{prop}\label{w2,1continuous}
Let \(G\subset\mathbb C\) be a domain. Every subharmonic function in \(W^{2,1}_{\mathrm{loc}}(G)\) is continuous.
\end{prop}
\begin{proof} 
By the standard endpoint Sobolev embedding in two real dimensions, such a function \(f\) admits a continuous representative \(v\). Since \(f\) is subharmonic, its values are recovered by limits of averages over shrinking balls. These averages coincide with those of \(v\), so the continuity of \(v\) gives \(f=v\) everywhere.
\end{proof}

\begin{prop}\label{prop:real-hessian-bound}
Let \(G\subset\mathbb C\) be a domain and let \(f\in W^{2,1}_{\mathrm{loc}}(G)\), identified with its continuous representative. For every disk \(B_R(p)\Subset G\), we have
\[
\int_{B_R(p)}|D^2f|\,dV
\ge
\int_0^{2\pi}f(p+Re^{\mathrm i\theta})\,d\theta-2\pi f(p),
\]
where \(D^2f\) is the real Hessian and \(|D^2f|\) is its Euclidean matrix norm. In particular, if \(f\) is invariant under rotations about \(p\), then
\[
\int_{B_R(p)}|D^2f|\,dV
\ge 2\pi\bigl(f(p+R)-f(p)\bigr).
\]
\end{prop}

\begin{proof}
By smooth approximation, we may assume that \(f\) is smooth near \(\overline{B_R(p)}\). Write \(z=p+re^{\mathrm i\theta}\) and \(\tau_\theta=(-\sin\theta,\cos\theta)\). In polar coordinates,
\[
D^2f(\tau_\theta,\tau_\theta)
=\frac{f_r}{r}+\frac{f_{\theta\theta}}{r^2}.
\]
Integrating over \(B_R(p)\) and using the periodicity in \(\theta\), we obtain
\[
\begin{aligned}
\int_{B_R(p)}D^2f(\tau_\theta,\tau_\theta)\,dV
&=\int_0^R\int_0^{2\pi}
\left(f_r+\frac{f_{\theta\theta}}r\right)\,d\theta\,dr\\
&=\int_0^{2\pi}f(p+Re^{\mathrm i\theta})\,d\theta-2\pi f(p).
\end{aligned}
\]
Since \(|\tau_\theta|=1\), we have \(|D^2f(\tau_\theta,\tau_\theta)|\leq |D^2f|\), which gives the desired inequality.
\end{proof}

For the use in Section \ref{sec:cantor-laplacian}, we briefly recall the Cantor probability measure, which we will use
to construct the solution $u_4$ whose Laplacian has a nonzero
singular part.

Let $C\subset[0,1]$ be the middle-third Cantor set, and set
$E=\lp2C-1\rp/4$ and $d=\log2/\log3\in(0,3/4)$.
Let $\nu$ be the Cantor probability measure on $E$, namely the unique
Borel probability measure supported on $E$ that assigns mass $2^{-n}$
to each of its $2^n$ basic intervals at level $n$, for every $n\ge0$.
This is the normalized $d$-dimensional Hausdorff measure on $E$.
Then $\operatorname{supp}\nu=E$.
Since $E$ has planar Lebesgue measure zero, $\nu$ is singular with
respect to planar Lebesgue measure.

As a reference, see \cite[Theorem~4.14]{Mattila1995}. There exist positive constants
$c_0$ and $C_0$ such that
\begin{equation}\label{cant:integral-measure-growth}
c_0r^d\le\nu(B(a,r))\le C_0r^d,
\qquad a\in E,\quad 0<r\le1.
\end{equation}
In particular, $\nu$ has no atoms.
After increasing $C_0$, we also have
$$
\nu(B(z,r))\le C_0r^d,
\qquad z\in\bC,\quad r>0.
$$

We recall the following standard properties of logarithmic and
Newtonian potentials; see \cite[Sections~4.2--4.3]{AG01}.

\begin{prop}\label{prelim:measure-poisson}
Let $n\ge2$ and let $\nu$ be a finite nonnegative Radon measure
with compact support in $\bR^n$.
For $x\ne0$, set
$$
J_n(x)=
\begin{cases}
\displaystyle-\frac1{2\pi}\log|x|, & n=2,\\[6pt]
\displaystyle\frac{|x|^{2-n}}{(n-2)\sigma_{n-1}}, & n\ge3,
\end{cases}
$$
where $\sigma_{n-1}$ is the surface area of the unit sphere
in $\bR^n$, and set $J_n(0)=+\infty$.
Then the potential
$$
U^\nu(x)=(J_n*\nu)(x)
=\int_{\bR^n}J_n(x-y)\,d\nu(y)
$$
belongs to $L^1_{\mathrm{loc}}(\bR^n)$.
It satisfies
$$
-\Delta U^\nu=\nu
$$
in the sense of distributions, where
$\Delta=\sum_{j=1}^n\partial^2/\partial x_j^2$,
and is smooth and harmonic on
$\bR^n\setminus\operatorname{supp}\nu$.

Consequently, $p=-U^\nu$ satisfies $\Delta p=\nu$
in the sense of distributions.
In particular, if $\nu$ is nonzero and singular with respect
to Lebesgue measure, then $\Delta p$ is a nonzero singular measure.
\end{prop}

For the Cantor measure $\nu$ introduced above, the logarithmic
potential
$$
g(z_1)=\int_E\log|z_1-a|\,d\nu(a)
$$
is subharmonic on $\bC$, harmonic on $\bC\setminus E$, and satisfies
$$
\Delta_{z_1}g=2\pi\nu
$$
in the sense of distributions.
Moreover, \eqref{cant:integral-measure-growth} gives
\begin{align}\label{log-est}
\int_{E\cap B(z,r)}|\log|z-a||\,d\nu(a)
\le Cr^d\lp1+|\log r|\rp,
\qquad z\in\bC,\quad 0<r<1.
\end{align}
This estimate implies that $g$ is continuous on $\bC$.

In real dimension three, the normalized Newton kernel is
$$
J_3(x)=\frac1{4\pi|x|}.
$$
For a finite nonnegative Radon measure $\mu$ with compact support
in $\bC\times\bR$, the corresponding unnormalized Newton potential
$$
v(z_1,t)
=\int_{\bC\times\bR}
\frac{1}{\sqrt{|z_1-a|^2+(t-s)^2}}\,d\mu(a,s)
$$
belongs to $L^1_{\mathrm{loc}}(\bC\times\bR)$ and satisfies
$$
\Delta_{z_1}v+v_{tt}=-4\pi\mu
$$
in the sense of distributions.
Moreover, $v$ is smooth and harmonic outside
$\operatorname{supp}\mu$.
These properties follow from
\cite[Sections~4.2--4.3]{AG01}.
In Section~\ref{sec:cantor-laplacian}, we will apply this construction
to a weighted product measure to obtain the function
used in the construction of $u_4$.

\section{The partial Legendre transformation}
\label{partial Legendre}

We first prove Theorems~\ref{main1} and~\ref{main:legendre-reductions}
for solutions independent of $\im z_2$ and anti-diagonal $\bS^1$-invariant solutions,  respectively,
and describe the corresponding local inverse transformations.
These transformations will be used to construct the continuous
and discontinuous solutions in the following sections. 

\subsection{Solutions independent of one real variable}\label{sec:harmonic-legendre}

\begin{proof}[Proof of Theorem~\ref{main1}]
Since each nonempty fiber
$\Om_{z_1}=\{t\in\bR\,|\,(z_1,t)\in\Om\}$ is an interval and
$v_{tt}>0$, the map $t\mapsto v_t(z_1,t)$ is strictly increasing
on $\Om_{z_1}$. Hence $\Phi_v$ is injective.
Since $v$ is smooth and $\det D\Phi_v=v_{tt}>0$,
the inverse function theorem shows that $\Phi_v$ is a
diffeomorphism onto its open image $\Phi_v(\Om)$.

Write $z_2=x_2+\sq y_2$. The normalized partial Legendre transform for $u$ is defined by
\begin{equation*}
x_2=v_t(z_1,t),\qquad
v^*(z_1,x_2)=x_2t-v(z_1,t),\qquad
u(z_1,z_2)=4v^*(z_1,\re z_2),
\end{equation*}
where $t=t_v(z_1,x_2)$. All derivatives of $v$ below are evaluated
at $(z_1,t_v(z_1,x_2))$. Differentiating $x_2=v_t(z_1,t)$, we obtain
$$
t_{x_2}=\frac1{v_{tt}},\qquad
t_{z_1}=-\frac{v_{tz_1}}{v_{tt}},\qquad
t_{\bar z_1}=-\frac{v_{t\bar z_1}}{v_{tt}}.
$$
Using these identities to differentiate $u=4(x_2t-v)$ gives
\begin{equation*}
\begin{split}
u_{z_1\bar z_1}
&=-4v_{z_1\bar z_1}+\frac{4v_{z_1t}v_{\bar z_1t}}{v_{tt}},\\
u_{z_1\bar z_2}
&=-\frac{2v_{z_1t}}{v_{tt}},\qquad
u_{z_2\bar z_2}=\frac1{v_{tt}}.
\end{split}
\end{equation*}
Since $v$ is harmonic, we have $4v_{z_1\bar z_1}+v_{tt}=0$.
Taking the determinant, we obtain
\begin{equation*}
\det(u_{j\bar k})=-\frac{4v_{z_1\bar z_1}}{v_{tt}}=1.
\end{equation*}
Together with $u_{z_2\bar z_2}=1/v_{tt}>0$, this proves that $u$
is strictly plurisubharmonic. Thus $u$ is a smooth solution
of \eqref{CMAdensity} on $F=\Phi_v(\Om)\times\bR$.
By construction, $u$ is independent of $y_2$.

We now extend the solution $u$ constructed above across $S$.
Since $F$ is open, $S=E\setminus F$ is relatively closed in $E$.
By assumption, $S$ is locally complete pluripolar in $E$,
and $u$ is locally bounded on $F$ near every point of $S$.
Lemma~\ref{lem:pluripolar-extension} therefore gives a unique
extension, still denoted by $u$, belonging to
$PSH(E)\cap L^\infty_{\mathrm{loc}}(E)$.
This extension satisfies \eqref{CMApluipotential} on $E$
in the Bedford--Taylor sense.
By uniqueness, the extension remains independent of
$\operatorname{Im}z_2$.
This completes the proof.
\end{proof}

We next describe the local inverse transformation.
Write $z_2=x_2+\sq y_2$, and let $u$ be a real-valued $C^2$ function
independent of $y_2$, with $u_{x_2x_2}>0$.
By the inverse function theorem, we can locally define
\begin{equation}\label{eq:cylindrical-inverse-legendre}
t=\frac14u_{x_2}(z_1,x_2),\qquad
v(z_1,t)=x_2t-\frac14u(z_1,x_2),
\end{equation}
where $x_2=x_2(z_1,t)$ is determined by the first identity.
All derivatives of $u$ below are evaluated at $(z_1,x_2(z_1,t))$.
Differentiating, we obtain
$$
v_t=x_2,\qquad
v_{tt}=\frac4{u_{x_2x_2}}>0,\qquad
v_{z_1\bar z_1}=-\frac14u_{z_1\bar z_1}
+\frac{u_{z_1x_2}u_{\bar z_1x_2}}{4u_{x_2x_2}}.
$$
Since $u$ is independent of $y_2$, we also have
$$
u_{z_2\bar z_2}=\frac14u_{x_2x_2},\qquad
u_{z_1\bar z_2}=\frac12u_{z_1x_2}.
$$
Combining these identities gives
$$
4v_{z_1\bar z_1}+v_{tt}
=\frac4{u_{x_2x_2}}\lp1-\det(u_{j\bar k})\rp.
$$
Hence $\det(u_{j\bar k})=1$ if and only if
$4v_{z_1\bar z_1}+v_{tt}=0$.

\subsection{Anti-diagonal \(\mathbb S^1\)-invariant solutions}

\begin{proof}[Proof of Theorem~\ref{main:legendre-reductions}]
On $F=\Phi_v(\Om)\times\bR$, define
$$
u(z_1,z_2)=4v^*(z_1,\re z_2).
$$
The proof of Theorem~\ref{main1} shows that $u$ is a smooth,
strictly plurisubharmonic solution of \eqref{CMAdensity}.
Choose a local branch of $\log z_2$ in the map $\Pi$
defined in the Introduction. Then
$$
\hat u=u\circ\Pi.
$$
Since $u$ depends on $z_2$ only through $\re z_2$, this expression
is independent of the choice of branch. 
Moreover, $\Pi$ is locally biholomorphic, with
$$
\det D_{\bC}\Pi=1,
$$
where $D_{\bC}\Pi=(\partial\Pi_i/\partial z_j)_{i,j=1}^{2}$
denotes its complex Jacobian matrix. Hence $\hat u$ is smooth and strictly plurisubharmonic on $\hat F$,
and
$$
\det(\hat u_{j\bar k})
=\left|\det D_{\bC}\Pi\right|^2
\bigl(\det(u_{j\bar k})\circ\Pi\bigr)
=1.
$$
Both $z_1z_2$ and $\log |z_2|$ are invariant under the anti-diagonal $\bS^1$-action. Thus $\hat F$ and $\hat u$ are invariant under this action.

Since $\hat F$ is open, $\hat S=\hat E\setminus\hat F$ is
relatively closed in $\hat E$.
By assumption, $\hat S$ is locally complete pluripolar in $\hat E$,
and $\hat u$ is locally bounded on $\hat F$ near every point
of $\hat S$.
Lemma~\ref{lem:pluripolar-extension} therefore gives a unique
extension, still denoted by $\hat u$, belonging to
$PSH(\hat E)\cap L^\infty_{\mathrm{loc}}(\hat E)$.
This extension satisfies \eqref{CMApluipotential} on $\hat E$
in the Bedford--Taylor sense.
The set $\hat E$ is also invariant under the anti-diagonal $\bS^1$-action,
so uniqueness implies that the extension is anti-diagonal $\bS^1$-invariant.
This completes the proof.
\end{proof}

We next describe the local inverse transformation. Define the holomorphic map
\[
\Psi:\bC^2\longrightarrow\bC\times\bC^*,
\qquad
\Psi(z_1,z_2)=\bigl(z_1e^{-z_2},e^{z_2}\bigr).
\]
For any open interval $J\subset\bR$ of length less than $2\pi$,
the restriction of $\Psi$ to $\bC\times(\bR+\sq J)$ is
biholomorphic onto its image, with inverse given by a suitable
branch of $\Pi$. Moreover, $\det D_{\bC}\Psi=1$.

Let $\hat u$ be a smooth, strictly plurisubharmonic
solution of \eqref{CMAdensity} invariant under the anti-diagonal
$\bS^1$-action, and define $u=\hat u\circ\Psi$ on the
corresponding coordinate neighborhood. The same change of variables shows that $u$ is strictly
plurisubharmonic and satisfies $\det(u_{j\bar k})=1$.
Write $z_2=x_2+\sq y_2$.
In these coordinates, the anti-diagonal $\bS^1$-action changes only $y_2$,
so $u$ is independent of $y_2$.
Consequently,
$$
u_{x_2x_2}=4u_{z_2\bar z_2}>0.
$$
After shrinking the neighborhood if necessary, we can define
the local inverse transformation by
$$
t=\frac14u_{x_2}(z_1,x_2),\qquad
v(z_1,t)=x_2t-\frac14u(z_1,x_2),
$$
where $x_2=x_2(z_1,t)$ is determined by the first identity.
The inverse calculation following the proof of
Theorem~\ref{main1} gives
$$
v_{tt}=\frac4{u_{x_2x_2}}>0,\qquad
4v_{z_1\bar z_1}+v_{tt}=0.
$$
Thus $v$ is harmonic.
Returning to the original coordinates, we obtain
$$
\hat u(z_1,z_2)=4v^*\lp z_1z_2,\log |z_2|\rp.
$$

\section{Singular solutions from logarithmic poles}
\label{sec:logarithmic-poles}

In this section, we apply the partial Legendre transformation to construct singular solutions from logarithmic poles. We first recall the single-pole solution, then superpose infinitely many poles to obtain a continuous solution outside $W^{2,1}_{\mathrm{loc}}$, and finally use a dense family of poles to construct a bounded solution that is discontinuous almost everywhere.

\subsection{A single logarithmic pole model}
\label{sec:single-source-nondini}
Consider the negative of the two-dimensional fundamental solution:
\[
b_0(z_1)=-\log |z_1|.
\]
It is positive and harmonic on $\bD\setminus\{0\}$.
Set
$$
h_0(z_1)=-2|z_1|^2(\log |z_1|-1),
$$
so that $(h_0)_{z_1\bar z_1}=2b_0$.
By Theorem~\ref{main1}, the harmonic function
$$
v_{\mathrm{sp}}(z_1,t)
=b_0(z_1)t^2-\frac14h_0(z_1)
$$
corresponds under the normalized partial Legendre transformation to
$$
U_{\mathrm{sp}}(z_1,z_2)
=\frac{x_2^2}{b_0(z_1)}+h_0(z_1).
$$
Since $b_0^{-1}$ and $h_0$ extend continuously by zero at $z_1=0$,
the same formula defines a continuous extension of $U_{\mathrm{sp}}$
to $\bD\times\bC$, with $U_{\mathrm{sp}}(0,z_2)=0$.
Lemma~\ref{lem:pluripolar-extension} shows that this extension
satisfies \eqref{CMApluipotential} in the Bedford--Taylor sense.
This is precisely the solution \eqref{notdini}.

We call the singularity of $b_0$ at $0$ a logarithmic pole.
By superposing logarithmic poles at suitable points, we will construct
the solutions in Theorems~\ref{theo:nonproper-w21-main}
and~\ref{theo:dense-poles-main}.

\subsection{Harmonic functions quadratic in one variable}
The preceding model uses a harmonic function that is quadratic in $t$.
The following proposition gives the corresponding construction
for a general positive harmonic coefficient.

\begin{prop}\label{cor:reciprocal-harmonic}
Let $D\subset\bC$ be a domain and let $b,h\in C^2(D)$ satisfy
\begin{equation*}
b>0,\qquad b_{z_1\bar z_1}=0,\qquad h_{z_1\bar z_1}=2b.
\end{equation*}
The function
\begin{equation}\label{eq:reciprocal-harmonic-primal}
u(z_1,z_2)=h(z_1)+\frac{(\re z_2)^2}{b(z_1)}
\end{equation}
is a smooth, strictly plurisubharmonic solution of
\eqref{CMAdensity} on $D\times\bC$.
Suppose further that $D=G\setminus A$, where $G\subset\bC$
is a domain and $A$ is a closed locally complete pluripolar
subset of $G$.
If $h$ and $b^{-1}$ are locally bounded near $A$, then $u$
extends uniquely to a locally bounded
plurisubharmonic solution of \eqref{CMApluipotential}
in the Bedford--Taylor sense on $G\times\bC$.
If $h$ and $b^{-1}$ extend continuously to $G$, then the
extension is continuous and given by the same formula
with the extended coefficients.
\end{prop}

\begin{proof}
Consider the quadratic function
\begin{equation*}
v(z_1,t)=b(z_1)t^2-\frac14h(z_1),
\qquad (z_1,t)\in D\times\bR.
\end{equation*}
We have
$$
v_{tt}=2b>0,\qquad
4v_{z_1\bar z_1}+v_{tt}=-h_{z_1\bar z_1}+2b=0.
$$
Thus $v$ is harmonic and strictly convex in $t$.
For each fixed $z_1$, the map $t\mapsto v_t(z_1,t)$ maps $\bR$
bijectively onto itself, with
\begin{equation}\label{eq:reciprocal-harmonic-fibre}
x_2=v_t=2bt,\qquad t=\frac{x_2}{2b}.
\end{equation}
Substituting \eqref{eq:reciprocal-harmonic-fibre} into $4(x_2t-v)$
and taking $x_2=\re z_2$ gives \eqref{eq:reciprocal-harmonic-primal}.
The stated properties of $u$ follow from Theorem~\ref{main1}.

The bounds on $h$ and $b^{-1}$ imply that $u$ is locally bounded near $A\times\mathbb C$.
Lemma~\ref{lem:pluripolar-extension} therefore gives the
unique locally bounded plurisubharmonic extension and
the equation in the Bedford--Taylor sense. If $h$ and $b^{-1}$ extend continuously to $G$, the same
formula defines a continuous extension of $u$.
By the same lemma, this agrees with the plurisubharmonic extension.
\end{proof}

\subsection{A continuous solution outside \(W^{2,1}_{\mathrm{loc}}\)}
\label{sec:old-w21}
We superpose a logarithmic pole at \(0\) with infinitely many logarithmic poles converging to \(0\). The pole at \(0\) ensures continuity at the accumulation point, while the weights of the remaining poles are chosen so that the solution lies outside \(W^{2,1}_{\mathrm{loc}}\).
\subsubsection{Construction}
Let \(p_k=2^{-k}\) for \(k\geq4\), and set \(E_1=\{0,p_4,p_5,\ldots\}\). For $p\in E_1$, define on $\bD\setminus\{p\}$
$$
L_p(z_1)=\log\frac2{|z_1-p|},\qquad
R_p(z_1)=|z_1-p|^2\lp L_p(z_1)+1\rp.
$$
Then $L_p$ is positive and harmonic, and
$(R_p)_{z_1\bar z_1}=L_p$ off $p$. Set $R_p(p)=0$. Then each $R_p$ belongs to $C^1(\bD)$. Since \(|z_1-p|<17/16\), the functions \(R_p\) and their first derivatives are uniformly bounded on \(\mathbb D\), independently of \(p\in E_1\).

Let \(\varepsilon_k=2^{-k}\) for \(k\geq4\). Define \(b\) on \(\mathbb D\setminus E_1\) and \(h\) on \(\mathbb D\) by
\begin{equation*}
b(z_1)=L_0(z_1)+\sum_{k=4}^{\infty}\eps_kL_{p_k}(z_1),\qquad
h(z_1)=2R_0(z_1)+2\sum_{k=4}^{\infty}\eps_kR_{p_k}(z_1).
\end{equation*}
Both series converge, together with all their derivatives,
uniformly on compact subsets of $\mathbb D\setminus E_1$.
Hence
$$
b>0,\qquad b_{z_1\bar z_1}=0,\qquad h_{z_1\bar z_1}=2b
\quad\text{on }\mathbb D\setminus E_1.
$$
Thus the function
\begin{equation*}
v_1(z_1,t)=b(z_1)t^2-\frac14h(z_1),
\qquad (z_1,t)\in(\mathbb D\setminus E_1)\times\bR,
\end{equation*}
is harmonic.
Applying Proposition~\ref{cor:reciprocal-harmonic} gives
\begin{equation*}
u_1(z_1,z_2)
=\frac{x_2^2}{b(z_1)}+h(z_1),
\qquad x_2=\re z_2.
\end{equation*}
We extend \(u_1\) to \(\mathbb D\times\mathbb C\) by setting \(b^{-1}=0\) on \(E_1\), so that
\[
u_1(p,z_2)=h(p),\qquad (p,z_2)\in E_1\times\mathbb C.
\]
The extension remains independent of \(y_2\).

\subsubsection{Verification of continuity and the equation}

The uniform bounds on \(R_p\) and its first derivatives, together with \(\sum_{k=4}^{\infty}\varepsilon_k<\infty\), imply that \(h\in C^1(\mathbb D)\). Since \(b\geq L_0\) and \(b\geq\varepsilon_kL_{p_k}\) for every \(k\geq4\), we have \(b(z_1)\to+\infty\) as \(z_1\to p\) for every \(p\in E_1\). Hence \(b^{-1}\) extends continuously by zero on \(E_1\), and \(u_1\in C(\mathbb D\times\mathbb C)\).

To apply the extension statement of Proposition \ref{cor:reciprocal-harmonic}, we first show that \(E_1\) is complete pluripolar. Define on \(\mathbb D\)
\[
\psi(z_1)=\log|z_1|+\sum_{k=4}^{\infty}
2^{-k}\log\frac{|z_1-p_k|}{2}.
\]
The partial sums decrease to \(\psi\) and converge locally uniformly off \(E_1\). Thus \(\psi\) is subharmonic on \(\mathbb D\) and satisfies \(\{\psi=-\infty\}=E_1\), which proves that \(E_1\) is complete pluripolar.

Since \(E_1\) is closed and both \(h\) and \(b^{-1}\) extend continuously to \(\mathbb D\), Proposition ~\ref{cor:reciprocal-harmonic}, with \(G=\mathbb D\), \(A=E_1\), gives
\[
u_1\in PSH(\mathbb D\times\mathbb C)\cap C(\mathbb D\times\mathbb C)
\]
and \((dd^cu_1)^2=\beta_2\) in the Bedford--Taylor sense. By Proposition \ref{prop:bt-viscosity}, \(u_1\) is also a viscosity solution.

\subsubsection{Failure of Dini continuity and $W^{2,1}_{\mathrm{loc}}$ regularity}
For \(k\geq4\), set \(D_k=B_{2^{-k-3}}(p_k)\)\(\subset\mathbb C\). We first establish a lower bound for \(u_1(z_1,z_2)-u_1(p_k,z_2)\), uniformly for \(z_1\in\partial D_k\) and \(z_2\in\mathbb C\) with \(x_2=\re z_2\in[1,2]\). This estimate will be used in both proofs. 

In the estimates below, \(c\) and \(C\) denote positive constants independent of \(j,k,z_1,z_2\), and may change from line to line. The notation \(O(1)\) denotes a quantity bounded uniformly for \(k\geq4\) and \(z_1\in\partial D_k\).

Fix \(z_1\in\partial D_k\). For \(j\ne k\), we have
\[
\begin{aligned}
|z_1-p_j|&\geq |p_k-p_j|-|z_1-p_k|\\
&=\bigl(1-2^{-|j-k|}\bigr)2^{-\min\{j,k\}}-2^{-k-3}\\
&>2^{-2-\min\{j,k\}}.
\end{aligned}
\]
Hence \(0<L_{p_j}(z_1)\leq Cj\) for \(j\ne k\), and therefore
\[
0\leq\sum_{j\ne k}\varepsilon_jL_{p_j}(z_1)
\leq C\sum_{j=4}^{\infty} j\varepsilon_j<\infty.
\]
Since \(\frac78\,2^{-k}\leq|z_1|\leq\frac98\,2^{-k}\) and \(|z_1-p_k|=2^{-k-3}\), we also have
\[
L_0(z_1)=k\log2+O(1),
\qquad L_{p_k}(z_1)=(k+4)\log2.
\]
Combining these observations gives
\[
b(z_1)
=L_0(z_1)+\sum_{j\ne k}\varepsilon_jL_{p_j}(z_1)
+\varepsilon_k(k+4)\log2
=k\log2+O(1).
\]

Since \(h\in C^1(\mathbb D)\), its gradient is bounded near \(0\), so
\[
|h(z_1)-h(p_k)|\leq C|z_1-p_k|\leq C2^{-k}.
\]
Combining these estimates with \(u_1(p_k,z_2)=h(p_k)\), we obtain
\begin{equation}\label{eq:old-boundary-gap}
\begin{aligned}
u_1(z_1,z_2)-u_1(p_k,z_2)
&=\frac{x_2^2}{b(z_1)}+h(z_1)-h(p_k)\\
&\geq\frac{2c}{k}-C2^{-k}
\geq\frac{c}{k}.
\end{aligned}
\end{equation}
for all sufficiently large \(k\), uniformly for \(z_1\in\partial D_k\) and \(x_2\in[1,2]\).

We next show that $u_1$ is not Dini continuous at $(0,1)$.
For \(0<\rho<1\), define
$$
\omega(\rho)
=\sup_{|\zeta-(0,1)|\leq\rho}
|u_1(\zeta)-u_1(0,1)|.
$$
Both \((p_k,1)\) and \((p_k+2^{-k-3},1)\) lie within distance \(2^{1-k}\) of \((0,1)\). Choose \(k_0\geq4\) sufficiently large. For every \(k\geq k_0\), the triangle inequality and \eqref{eq:old-boundary-gap} give
\[
\begin{aligned}
2\omega(2^{1-k})
&\geq\left|u_1(p_k+2^{-k-3},1)-u_1(0,1)\right|
+\big|u_1(p_k,1)-u_1(0,1)\big|\\
&\geq u_1(p_k+2^{-k-3},1)-u_1(p_k,1)\\
&\geq\frac{c}{k}.
\end{aligned}
\]
Since \(\omega\) is nondecreasing, we have
\[
\begin{aligned}
\int_0^{2^{2-k_0}}\frac{\omega(\rho)}{\rho}\,d\rho
&=\sum_{k=k_0}^{\infty}
\int_{2^{1-k}}^{2^{2-k}}\frac{\omega(\rho)}{\rho}\,d\rho\\
&\geq(\log2)\sum_{k=k_0}^{\infty}\omega(2^{1-k})\\
&\geq\frac{c\log2}{2}\sum_{k=k_0}^{\infty}\frac1k
=+\infty.
\end{aligned}
\]
Thus $u_1$ is not Dini continuous at $(0,1)$.

Finally, we prove the failure of $W^{2,1}_{\mathrm{loc}}$ regularity.
Set
$$
Q=\{x_2+\sq y_2\,|\,1\leq x_2\leq2,\ |y_2|\leq 1\},
\qquad B_k=D_k\times Q.
$$
The sets $B_k$ are pairwise disjoint and lie in a fixed compact
subset of $\bD\times\bC$.
Suppose, to the contrary, that
$u_1\in W^{2,1}_{\mathrm{loc}}(\bD\times\bC)$.
By the Sobolev slicing property, \(u_1(\cdot,z_2)\) belongs to \(W^{2,1}_{\mathrm{loc}}(\mathbb D)\) for almost every \(z_2\in Q\). Applying Proposition \ref{prop:real-hessian-bound} to these slices on \(D_k\), integrating over \(Q\), and using \eqref{eq:old-boundary-gap}, we obtain for all sufficiently large \(k\)
\begin{equation}\label{eq:old-harmonic-series}
\begin{aligned}
\int_{B_k}|\D^2_{z_1}u_1|\,dV
&\geq\int_Q\int_0^{2\pi}
\lp u_1\big(p_k+2^{-k-3}e^{\sq\theta},z_2\big)
-u_1(p_k,z_2)\rp\,d\theta\,dV(z_2)\\
&\geq\frac{c}{k},
\end{aligned}
\end{equation}
where $\D^2_{z_1}$ is the real Hessian in $z_1$.
Since $|\D^2u_1|\geq|\D^2_{z_1}u_1|$, summing
\eqref{eq:old-harmonic-series} over the disjoint sets $B_k$ for all sufficiently large \(k\) contradicts the local integrability of $\D^2u_1$.
Therefore $u_1\notin W^{2,1}_{\mathrm{loc}}(\bD\times\bC)$.

\subsection{A bounded solution discontinuous almost everywhere}
\label{sec:logarithmic-regularity}

In this section, we prove Theorem~\ref{theo:dense-poles-main}. We first construct a locally bounded solution $u$ on \(\mathbb D\times\mathbb C\) by superposing logarithmic poles at a countable dense set of points in \(\mathbb D\). This solution is discontinuous almost everywhere and is not classically differentiable at any point where \(x_2\ne0\). It lies outside \(W^{2,1}_{\mathrm{loc}}\), although its distributional Laplacian belongs to \(L^1_{\mathrm{loc}}\). We then replace \(z_2\) by \(z_2+2\) and restrict the solution to \(\mathbb D^2\), obtaining the bounded, nowhere classically differentiable solution $u_2$ in Theorem~\ref{theo:dense-poles-main}.

\subsubsection{Construction}
Let \(A=\{p_k\}_{k\ge1}\subset\mathbb D\) be a countable dense set of distinct points. Choose positive weights \(\varepsilon_k\) with \(\sum_k\varepsilon_k<\infty\). As in Section~\ref{sec:old-w21}, for \(p\in\mathbb D\), set
\[
L_p(z_1)=\log\frac2{|z_1-p|},
\qquad
R_p(z_1)=|z_1-p|^2\bigl(L_p(z_1)+1\bigr),
\]
with \(L_p(p)=+\infty\) and \(R_p(p)=0\). Define
\[
b(z_1)=\sum_{k=1}^{\infty}\varepsilon_kL_{p_k}(z_1),
\qquad
h(z_1)=2\sum_{k=1}^{\infty}\varepsilon_kR_{p_k}(z_1).
\]
Set \(P=\{z_1\in\mathbb D\,|\,b(z_1)=+\infty\}\) and define \(b^{-1}=0\) on \(P\). We then define \(u\) on \(\mathbb D\times\mathbb C\) by
\[
u(z_1,z_2)=h(z_1)+\frac{x_2^2}{b(z_1)}.
\]

To prove that \(u\) satisfies the equation in the Bedford–Taylor sense, we approximate it using solutions associated with finitely many poles. For \(N\ge1\), let \(b_N\) and \(h_N\) denote the \(N\)th partial sums of the series defining \(b\) and \(h\), respectively. On \(\mathbb D\setminus\{p_1,\ldots,p_N\}\), we have
\[
(b_N)_{z_1\bar z_1}=0,
\qquad
(h_N)_{z_1\bar z_1}=2b_N.
\]
Moreover,
\[
b_N(z_1)\ge
\varepsilon_1\log\frac2{1+|p_1|}
=:c_0>0.
\]
Since \(h_N\) is continuous and \(b_N^{-1}\) extends continuously by zero at each pole, Proposition \ref{cor:reciprocal-harmonic} gives a continuous Bedford–Taylor solution
\[
u_N(z_1,z_2)
=h_N(z_1)+\frac{x_2^2}{b_N(z_1)}
\]
on \(\mathbb D\times\mathbb C\).
The bounds \(0\le R_p\le4\) and \(b_N\ge c_0\), together with the summability of the weights, show that the functions \(u_N\) are locally uniformly bounded. Moreover, since \(b_N\nearrow b\) and \(R_p\le4\),
\[
u_N+8\sum_{k>N}\varepsilon_k\searrow u.
\]
Adding a constant does not change the Monge–Ampère measure. Lemma \ref{convergence-theorem} therefore gives
\[
u\in PSH(\mathbb D\times\mathbb C)
\cap L^\infty_{\mathrm{loc}}(\mathbb D\times\mathbb C),
\qquad
(dd^cu)^2=\beta_2.
\]

\subsubsection{Singular behavior}
The functions \(R_p\) and their first derivatives are uniformly bounded on \(\mathbb D\), independently of \(p\in\mathbb D\). The summability of the weights therefore implies that \(h\in C^1(\mathbb D)\). Moreover, since \(\mathbb D\subset B_2(p)\), we have
\[
\int_{\mathbb D}L_p(z_1)^2\,dV(z_1)
\le 2\pi\int_0^2 r\left(\log\frac2r\right)^2\,dr
=2\pi.
\]
By Minkowski’s inequality,
\[
\|b\|_{L^2(\mathbb D)}
\le \sum_{k=1}^{\infty}\varepsilon_k
\|L_{p_k}\|_{L^2(\mathbb D)}
\le \sqrt{2\pi}\sum_{k=1}^{\infty}\varepsilon_k
<\infty.
\]
In particular, \(b\) is finite almost everywhere, so \(P\) has Lebesgue measure zero.

Next, we show that the continuity set of \(u\) is precisely
\[
(P\times\mathbb C)\cup
\bigl(\mathbb D\times \{0\}\times\mathbb R\bigr).
\]
Indeed, each \(b_N\) is lower semicontinuous, so \(b=\sup_N b_N\) is lower semicontinuous. Hence \(b(z_1)\to+\infty\) as \(z_1\to q\) for every \(q\in P\), and thus \(b^{-1}(z_1)\to0\). Together with the continuity of \(h\), this proves that \(u\) is continuous at every point of \(P\times\mathbb C\). Clearly, \(u\) is also continuous at every point where \(x_2=0\).

On the other hand, fix \(q\in\mathbb D\setminus P\) and \(z_2\in\mathbb C\) with \(x_2\ne0\). Since \(A\) is dense, there is a sequence of poles \(p_{k_j}\to q\). Then
\begin{equation*}
\begin{split}
\lim_{j\rightarrow\infty}u(p_{k_j},z_2)
=\lim_{j\rightarrow\infty}h(p_{k_j})=h(q)
<h(q)+\frac{x_2^2}{b(q)}
=u(q,z_2),
\end{split}
\end{equation*}
so \(u\) is discontinuous at \((q,z_2)\). Therefore, its discontinuity set is precisely 
\[
(\mathbb D\setminus P)\times
(\mathbb R\setminus\{0\})\times\mathbb R.
\]
Thus \(u\) is discontinuous almost everywhere.

\begin{rem}
A standard consequence of the Baire category theorem (see \cite[Theorem 48.5]{Munkres}) is that a pointwise limit of continuous functions on a Euclidean domain has a dense set of continuity points. Since a locally bounded plurisubharmonic function is locally a pointwise limit of smooth functions, its continuity points are dense. In particular, it cannot be discontinuous everywhere.
\end{rem}

We next prove that \(u\) is not classically differentiable at any point where \(x_2\ne0\). Fix \((q,z_2^0)\in\mathbb D\times\mathbb C\) with \(x^0_2=\re z_2^0\ne0\). If \(q\notin P\), then \(u\) is discontinuous at \((q,z_2^0)\), as proved above, and hence is not differentiable there.

Suppose that \(q\in P\) and that \(u\) is differentiable at \((q,z_2^0)\). Since \(u(q,z_2^0)=h(q)\), there exist \(C>0\) and \(r>0\), with \(B_r(q)\Subset\mathbb D\), such that, for every \(z_1\in B_r(q)\),
\[
\begin{aligned}
0\le \frac{(x_2^0)^2}{b(z_1)}
&=u(z_1,z_2^0)-h(z_1)\\
&=\bigl(u(z_1,z_2^0)-u(q,z_2^0)\bigr)
-\bigl(h(z_1)-h(q)\bigr)\\
&\le |u(z_1,z_2^0)-u(q,z_2^0)|
+|h(z_1)-h(q)|\\
&\le C|z_1-q|,
\end{aligned}
\]
where the last inequality follows from the differentiability of \(u\) at \((q,z_2^0)\) and the fact that \(h\in C^1(\mathbb D)\). Consequently, for every \(z_1\in B_r(q)\setminus\{q\}\),
\[
b(z_1)\ge \frac{(x_2^0)^2}{C|z_1-q|}.
\]
This contradicts \(b\in L^2(\mathbb D)\), since
\[
\int_{B_r(q)}b(z_1)^2\,dV(z_1)
\ge C^{-2}\int_{B_r(q)}\frac{dV(z_1)}{|z_1-q|^2}
=+\infty.
\]

We also have \(u\notin W^{2,1}_{\mathrm{loc}}(\mathbb D\times\mathbb C)\). Otherwise, by Sobolev slicing,
\(
u(\cdot,z_2)\in W^{2,1}_{\mathrm{loc}}(\mathbb D)
\)
for almost every \(z_2\in\mathbb C\). Each such slice is subharmonic and therefore continuous by Proposition \ref{w2,1continuous}. This contradicts the fact that every slice with \(x_2\ne0\) is discontinuous at every point of \(\mathbb D\setminus P\).

\section{A continuous exhausting solution outside \(W^{2,1}_{\mathrm{loc}}\)}
\label{sec:proper-w21}

In this section, we prove Theorem \ref{theo:proper-w21-main}. The Newton kernel
\((|z_1|^2+t^2)^{-1/2}\) is a natural candidate for \(v_{tt}\).
Integrating this kernel twice in \(t\), with integration terms
chosen to preserve harmonicity, yields, up to normalization,
the harmonic dual of Warren's solution, which we denote by \(h_W\).

As in Section \ref{sec:logarithmic-poles}, we use the superposition principle: we add to \(h_W\) positive weighted sums of its translates and reflections in \(t\). This preserves harmonicity and strict convexity in \(t\). The resulting function also grows superlinearly as \(t\to\pm\infty\), so its partial Legendre transform is defined for every value of the dual variable at each \(z_1\) outside the singular set. After continuous extension and holomorphic pullback, we obtain a continuous proper  solution in the Bedford–Taylor sense on \(\mathbb C^2\).

\subsection{Superposition of \texorpdfstring{Warren’s}{Warren’s} harmonic dual}

We write Warren's solution \eqref{War} as
$U_W(z_1,x_2)=e^{2x_2}+|z_1|^2e^{-2x_2}$,
where $x_2=\re z_2$.
For $z_1\ne0$, the inverse transformation
\eqref{eq:cylindrical-inverse-legendre} gives
$$
t=\frac14(U_W)_{x_2}
=\frac12\lp e^{2x_2}-|z_1|^2e^{-2x_2}\rp.
$$
Set $R=\sqrt{|z_1|^2+t^2}$.
Then $U_W=2R$ and $x_2=\frac12\log(t+R)$, so
\begin{equation*}
h_W(z_1,t)=x_2t-\frac14U_W(z_1,x_2)
=\frac12\lp t\log(t+R)-R\rp.
\end{equation*}
By the inverse transformation in Section~\ref{partial Legendre},
$h_W$ is harmonic on $(\bC\setminus\{0\})\times\bR$, with
$$
(h_W)_t=\frac12\log(t+R),\qquad
(h_W)_{tt}=\frac1{2R}>0.
$$
On the half-space \(t\geq0\), \(h_W\) extends continuously across \(z_1=0\), with
\[
h_W(0,t)=\frac12t\bigl(\log(2t)-1\bigr),
\]using the convention \(0\log0=0\). For each fixed \(t<0\), however, \(h_W(z_1,t)\to+\infty\) as \(z_1\to0\).

For $\xi\in\bC$ and $s\in\bR$, the function
$h_W(z_1-\xi,s-t)$ is harmonic and strictly convex in $t$
on $(\bC\setminus\{\xi\})\times\bR$, with
\begin{equation*}
\lp h_W(z_1-\xi,s-t)\rp_{tt}
=\frac1{2\sqrt{|z_1-\xi|^2+(s-t)^2}}.
\end{equation*}
On the half-space $t\leq s$, $h_W(z_1-\xi,s-t)$ extends
continuously to $z_1=\xi$, with
\begin{equation}\label{prw21:eq:capped-source}
h_W(0,s-t)=\frac{1}{2}(s-t)\lp\log\lp2(s-t)\rp-1\rp,
\qquad t\leq s.
\end{equation}
For each fixed $t>s$, however,
$h_W(z_1-\xi,s-t)\to+\infty$ as $z_1\to\xi$.

For \(j\geq1\) and \(k\geq4\), set $s_j=1+2^{-j}\in\mathbb R$, $\xi_k=1+2^{-k}\in\mathbb C$. Define
\begin{equation*}
\begin{aligned}
C(z_1,t)&=2\sum_{j=1}^{\infty}2^{-j}h_W(z_1-1,s_j-t),\\
S(z_1,t)&=2\sum_{k=4}^{\infty}3^{-k}h_W(z_1-\xi_k,1-t).
\end{aligned}
\end{equation*}
We then define the harmonic function used in our construction by
\begin{equation*}
v_3=h_W+C+S.
\end{equation*}
The accumulation of the levels $s_j$ in $C$ gives a
$1/\log(1/r)$ lower bound for the modulus of continuity,
so the resulting solution is not Dini continuous.
The terms in $S$ give nonsummable lower bounds for the integrals
of the real Hessian over disjoint regions,
so the solution does not belong to $W^{2,1}_{\mathrm{loc}}$.
We prove these assertions in the following subsections.

Let $E=\{1,\xi_4,\xi_5,\ldots\}$. The series defining $C$ and $S$ converge, together with all their derivatives,
uniformly on compact subsets of $(\bC\setminus E)\times\bR$.
Thus $v_3$ is harmonic on
$(\bC\setminus(E\cup\{0\}))\times\bR$, with
\begin{equation*}
\begin{aligned}
&(v_3)_{t}
=\frac{1}{2}\log\left(\sqrt{|z_1|^2+t^2}+t\right)
-\sum_{j=1}^{\infty}
2^{-j}\log\left(\sqrt{|z_1-1|^2+(s_j-t)^2}+(s_j-t)\right)\\
&\hphantom{(v_3)_{tt}=}\ \ {-}
\sum_{k=4}^{\infty}
3^{-k}\log\left(\sqrt{|z_1-\xi_k|^2+(1-t)^2}+1-t\right),
\end{aligned}
\end{equation*}
and
\begin{equation}\label{prw21:eq:total-harmonic}
\begin{aligned}
&(v_3)_{tt}
=\frac1{2\sqrt{|z_1|^2+t^2}}
+\sum_{j=1}^{\infty}
\frac{2^{-j}}{\sqrt{|z_1-1|^2+(s_j-t)^2}}\\
&\hphantom{(v_3)_{tt}=}\ \ +
\sum_{k=4}^{\infty}
\frac{3^{-k}}{\sqrt{|z_1-\xi_k|^2+(1-t)^2}}>0.
\end{aligned}
\end{equation}
For each $z_1\notin E\cup\{0\}$, the convex function
$(C+S)(z_1,\cdot)$ has an affine lower bound.
Since $h_W(z_1,\cdot)$ grows superlinearly, we have
$v_3(z_1,t)/|t|\to+\infty$ as $t\to\pm\infty$.
Thus, for every \(x_2\in\mathbb R\), the supremum defining
\begin{equation}\label{prw21:eq:cylindrical-potential}
U_3(z_1,x_2)=4\sup_{t\in\mathbb R}
\left\{x_2t-v_3(z_1,t)\right\}
\end{equation}
is attained at a unique \(t\), determined by
$(v_3)_t(z_1,t)=x_2$.
By Theorem~\ref{main:legendre-reductions}, the function
\begin{equation}\label{prw21:eq:entire-pullback}
u_3(\zeta_1,\zeta_2)
=U_3\lp\zeta_1\zeta_2,\log|\zeta_2|\rp,
\qquad \zeta_1\zeta_2\notin E\cup\{0\},
\end{equation}
is a smooth strictly plurisubharmonic solution of
\eqref{CMAdensity} on this domain.

\subsection{Continuous extension and verification of the equation}
\label{prw21:step:pluripotential}

The sequences $\xi_k$ and $s_j$ are bounded, and
$\sum_{j=1}^{\infty}2^{-j}+\sum_{k=4}^{\infty}3^{-k}<\infty$.
By \eqref{prw21:eq:capped-source}, the series defining $C$ and $S$
converge locally uniformly on $\bC\times(-\infty,1]$.
Hence $C$ and $S$ are continuous on $\bC\times(-\infty,1]$
and smooth on $\bC\times(-\infty,1)$. Since \(h_W\) is smooth on \((\mathbb C\setminus\{0\})\times\mathbb R\), \(v_3\) extends continuously to \((\mathbb C\setminus\{0\})\times(-\infty,1]\), and this extension is smooth on \((\mathbb C\setminus\{0\})\times(-\infty,1)\). For \(\xi\in E\) and \(x_2\in\mathbb R\), define
\begin{equation}\label{prw21:eq:capped-conjugate}
U_3(\xi,x_2)=4\sup_{t\leq1}\left\{x_2t-v_3(\xi,t)\right\}.
\end{equation}
Since $v_3(\xi,\cdot)$ is strictly convex and grows superlinearly
as $t\to-\infty$, the supremum is attained at a unique $t\leq1$.

We first consider the isolated points \(\xi_k\). At \(z_1=\xi_k\), the \(k\)-th term of \(S\) satisfies
\[
\frac{d}{dt}\left(2\cdot3^{-k}h_W(0,1-t)\right)
=-3^{-k}\log\bigl(2(1-t)\bigr)
\longrightarrow+\infty
\qquad\text{as }t\to1^-.
\]
Since \(\xi_k\) is separated from all other points in $E$, the remaining terms define a smooth function near \((\xi_k,1)\). Therefore
\[
(v_3)_t(\xi_k,t)\longrightarrow+\infty
\qquad\text{as }t\to1^-,
\]
and the supremum defining \(U_3(\xi_k,x_2)\) is attained at \(t<1\) for every \(x_2\in\mathbb R\).

Fix \(x_2^0\in\mathbb R\), and let \(t_0<1\) be the unique value attaining the supremum defining \(U_3(\xi_k,x_2^0)\). Then
\[
(v_3)_t(\xi_k,t_0)=x_2^0.
\]
Since \(t_0<1\), \(v_3\) is smooth near \((\xi_k,t_0)\), with \((v_3)_{tt}(\xi_k,t_0)>0\). The implicit function theorem therefore gives a smooth function \(t=t(z_1,x_2)<1\) near \((\xi_k,x_2^0)\), satisfying
\[
(v_3)_t\bigl(z_1,t(z_1,x_2)\bigr)=x_2.
\]
By strict convexity, \(t(z_1,x_2)\) attains the supremum over \(\mathbb R\) when \(z_1\notin E\), and over \(t\le1\) when \(z_1\in E\). Thus the original definition of \(U_3\) on \((\mathbb C\setminus(E\cup\{0\}))\times\mathbb R\) and the above definition on \(E\times\mathbb R\) both yield
\[
U_3(z_1,x_2)
=4\left(x_2t(z_1,x_2)
-v_3\bigl(z_1,t(z_1,x_2)\bigr)\right)
\]
throughout a neighborhood of \((\xi_k,x_2^0)\). Since \(x_2^0\) was arbitrary, this proves that \(U_3\) is smooth across \(\{\xi_k\}\times\mathbb R\).

We next consider \((1,x_2^0)\), where \(x_2^0\in\mathbb R\). We first compute the endpoint slope to determine whether the supremum defining \(U_3(1,x_2^0)\) is attained at \(t<1\) or at \(t=1\). For \(t<1\),
\[
\begin{aligned}
C_t(1,t)
&=-\sum_{j=1}^{\infty}
2^{-j}\log\bigl(2(s_j-t)\bigr),\\
S_t(1,t)
&=-\sum_{k=4}^{\infty}
3^{-k}
\log\left(1-t+\sqrt{2^{-2k}+(1-t)^2}\right).
\end{aligned}
\]
For \(0\le t<1\), the summands in these two series are bounded in absolute value by constant multiples of \(j2^{-j}\) and \(k3^{-k}\), respectively. These bounds are summable and independent of \(t\), so we may take the limit term by term. Using
\(
\sum_{j=1}^{\infty}2^{-j}=1\)
and 
\(
\sum_{j=1}^{\infty}j2^{-j}=2,
\)
we obtain
\begin{align}\label{endpoint slope}
 x_*:=\lim_{t\to1^-}(v_3)_t(1,t)
=(h_W)_t(1,1)+\log2
+(\log2)\sum_{k=4}^{\infty}k3^{-k}<\infty.
\end{align}
Since \((v_3)_t(1,\cdot)\) is strictly increasing, the supremum defining \(U_3(1,x_2^0)\) is attained at \(t<1\) when \(x_2^0<x_*\), and at \(t=1\) when \(x_2^0\ge x_*\).
Suppose first that \(x_2^0<x_*\), and let \(t_0<1\) be the unique value attaining the supremum. Then
\[
(v_3)_t(1,t_0)=x_2^0,
\qquad
(v_3)_{tt}(1,t_0)>0.
\]
Since \(v_3\) is smooth near \((1,t_0)\), the same implicit function argument used at \(\xi_k\) gives a smooth function \(t=t(z_1,x_2)<1\) near \((1,x_2^0)\), with
\[
U_3(z_1,x_2)
=4\left(x_2t(z_1,x_2)
-v_3\bigl(z_1,t(z_1,x_2)\bigr)\right).
\]
Thus \(U_3\) is smooth near \((1,x_2^0)\).

Now suppose that \(x_2^0\ge x_*\). In this case,
\[
U_3(1,x_2^0)=4\bigl(x_2^0-v_3(1,1)\bigr).
\]
We prove continuity at \((1,x_2^0)\) by considering an arbitrary sequence
\(
(z_1^{(m)},x_2^{(m)})\longrightarrow(1,x_2^0)
\)
in \((\mathbb C\setminus\{0\})\times\mathbb R\).
By the definition of \(U_3\), the choice \(t=1\) gives
\[
U_3(z_1^{(m)},x_2^{(m)})
\ge4\bigl(x_2^{(m)}-v_3(z_1^{(m)},1)\bigr).
\]
The continuity of \(v_3\) on \(t\le1\) therefore gives
\[
\liminf_{m\to\infty}U_3(z_1^{(m)},x_2^{(m)})
\ge4\bigl(x_2^0-v_3(1,1)\bigr)
=U_3(1,x_2^0).
\]
It remains to estimate the upper limit of \(U_3(z_1^{(m)},x_2^{(m)})\). For each \(m\), let \(t_m\) be the unique value at which the supremum defining \(U_3(z_1^{(m)},x_2^{(m)})\) is attained. Thus
\[
U_3(z_1^{(m)},x_2^{(m)})
=4\bigl(x_2^{(m)}t_m-v_3(z_1^{(m)},t_m)\bigr),
\]
and \(t_m\le1\) whenever \(z_1^{(m)}\in E\). 

We first show that \((t_m)\) is bounded. For all sufficiently large \(m\), we have \(|z_1^{(m)}-1|\le1/4\). For each \(z_1\ne0\), the function \(h_W(z_1,\cdot)\) attains its minimum at \(t=(1-|z_1|^2)/2\), so
\begin{equation*}
h_W(z_1,t)\ge-\frac{1+|z_1|^2}{4},
\qquad z_1\ne0,\quad t\in\mathbb R.
\end{equation*}
Direct substitution shows that this bound also holds when \(z_1=0\) and \(t\ge0\). Since \(\{\xi_k\}\) is bounded and the weights are summable, there is a constant \(B>0\) such that
\[
(C+S)(z_1,t)\ge-B
\]
for \(|z_1-1|\le1/4\), with \(t\in\mathbb R\) if \(z_1\notin E\) and \(t\le1\) if \(z_1\in E\). The same bound holds after any one summand is removed. Moreover, \(h_W(z_1,t)\) grows superlinearly as \(|t|\to\infty\), uniformly for \(|z_1-1|\le1/4\). Since \(\{x_2^{(m)}\}\) is bounded, an unbounded subsequence of \(\{t_m\}\) would give
\[
x_2^{(m)}t_m-v_3\bigl(z_1^{(m)},t_m\bigr)
\le x_2^{(m)}t_m-h_W\bigl(z_1^{(m)},t_m\bigr)+B
\longrightarrow-\infty,
\]
contrary to the preceding lower bound for \(U_3(z_1^{(m)},x_2^{(m)})\). Hence \(\{t_m\}\) is bounded.

We next show that \(t_m\to1\). Pass to any convergent subsequence, still indexed by \(m\), and write \(t_m\to\tau\). If \(\tau>1\), then \(t_m>1\), and hence \(z_1^{(m)}\notin E\), for all sufficiently large \(m\). Choose an index \(\ell\) such that \(s_\ell<\tau\). Since
\[
z_1^{(m)}-1\longrightarrow0,
\qquad
s_\ell-t_m\longrightarrow s_\ell-\tau<0,
\]
the explicit formula for \(h_W\) gives
\[
2\cdot2^{-\ell}
h_W\bigl(z_1^{(m)}-1,s_\ell-t_m\bigr)
\longrightarrow+\infty.
\]
The remaining part of \(v_3(z_1^{(m)},t_m)\) has a uniform lower bound, so
\[
v_3(z_1^{(m)},t_m)\longrightarrow+\infty.
\]
Since \(x_2^{(m)}t_m\) remains bounded, this again contradicts the lower bound for \(U_3(z_1^{(m)},x_2^{(m)})\).
If \(\tau<1\), continuity of \(v_3\) at \((1,\tau)\) gives
\[
U_3(z_1^{(m)},x_2^{(m)})
\longrightarrow
4\bigl(x_2^0\tau-v_3(1,\tau)\bigr)
<U_3(1,x_2^0).
\]
The strict inequality follows because \(t=1\) is the unique value attaining the supremum defining \(U_3(1,x_2^0)\). This contradicts the lower limit inequality above. Thus every convergent subsequence of \(\{t_m\}\) has limit \(1\), and consequently \(t_m\to1\).

We now estimate the lower limit of \(v_3(z_1^{(m)},t_m)\). Along the indices for which \(t_m\le1\), continuity on \(t\le1\) gives
\(
v_3(z_1^{(m)},t_m)\rightarrow v_3(1,1).
\)
For those indices with \(t_m>1\), convexity gives
\[
v_3(z_1^{(m)},t_m)
\ge v_3(z_1^{(m)},1)+(t_m-1)
\bigl(v_3(z_1^{(m)},1)-v_3(z_1^{(m)},0)\bigr).
\]
The right-hand side tends to \(v_3(1,1)\), since \(t_m\to1\) and the values of \(v_3\) at \(t=0,1\) depend continuously on \(z_1\). Therefore,
\[
\liminf_{m\to\infty}v_3(z_1^{(m)},t_m)\ge v_3(1,1).
\]
Using \(x_2^{(m)}t_m\to x_2^0\), we conclude that
\[
\limsup_{m\to\infty}U_3(z_1^{(m)},x_2^{(m)})
\le4\bigl(x_2^0-v_3(1,1)\bigr)
=U_3(1,x_2^0).
\]
Together with the lower limit inequality, this proves continuity at \((1,x_2^0)\). Hence \(U_3\) is continuous on \((\mathbb C\setminus\{0\})\times\mathbb R\).

Using this extension, we extend \(u_3\) continuously to all points with \(\zeta_1\zeta_2\ne0\) by
\[
u_3(\zeta_1,\zeta_2)
=U_3\bigl(\zeta_1\zeta_2,\log|\zeta_2|\bigr).
\]
We next extend \(u_3\) smoothly across the coordinate axes.
The functions \(C\) and \(S\) are smooth on \(\{|z_1|<1/2\}\times\mathbb R\). We retain the original variables \((z_1,x_2)\) for \(U_3\) and \((z_1,t)\) for \(v_3\). For \(0<|\zeta_1\zeta_2|<1/2\), the pullback gives \(z_1=\zeta_1\zeta_2\) and \(x_2=\log|\zeta_2|\). The Legendre relation \(x_2=(v_3)_t(z_1,t)\) therefore yields
\begin{align}\label{tag-1}
\log|\zeta_2|
=\frac12\log(t+R)+C_t+S_t,
\end{align}
where \(R=\sqrt{|\zeta_1\zeta_2|^2+t^2}\). Here and throughout the axis extension argument, \(C\), \(S\), and their derivatives are evaluated at \((\zeta_1\zeta_2,t)\). Since \(t(h_W)_t-h_W=R/2\), the identity \(U_3=4(x_2t-v_3)\) gives
\begin{align}\label{tag-2}
u_3=2R+4\bigl(t(C_t+S_t)-C-S\bigr).
\end{align}
Exponentiating \eqref{tag-1} and using \((R-t)(R+t)=|\zeta_1\zeta_2|^2\), we obtain
\[
R+t=|\zeta_2|^2e^{-2(C_t+S_t)},
\qquad
R-t=|\zeta_1|^2e^{2(C_t+S_t)}.
\]
Taking the difference and sum of these identities and using \eqref{tag-2}, we obtain
\begin{align}\label{tag-3}
2t
=|\zeta_2|^2e^{-2(C_t+S_t)}
-|\zeta_1|^2e^{2(C_t+S_t)},
\end{align}
\begin{align}\label{tag-4}
u_3
=|\zeta_2|^2e^{-2(C_t+S_t)}
+|\zeta_1|^2e^{2(C_t+S_t)}
+4\bigl(t(C_t+S_t)-C-S\bigr).
\end{align}
The right-hand sides of \eqref{tag-3} and \eqref{tag-4} are smooth for \(|\zeta_1\zeta_2|<1/2\) and \(t\in\mathbb R\). On this domain, define
\[
F(\zeta_1,\zeta_2,t)
=2t-|\zeta_2|^2e^{-2(C_t+S_t)}
+|\zeta_1|^2e^{2(C_t+S_t)}.
\]
Since \(C_{tt}+S_{tt}\ge0\), we have
\[
\frac{\partial F}{\partial t}
=2+2(C_{tt}+S_{tt})
\left(
|\zeta_2|^2e^{-2(C_t+S_t)}
+|\zeta_1|^2e^{2(C_t+S_t)}
\right)\ge2.
\]
For each fixed \((\zeta_1,\zeta_2)\), the bound \(\partial F/\partial t\ge2\) implies that \(F(\zeta_1,\zeta_2,\cdot)\) is strictly increasing and tends to \(\pm\infty\) as \(t\to\pm\infty\).
Hence the equation \(F(\zeta_1,\zeta_2,t)=0\) has a unique solution \(t=t(\zeta_1,\zeta_2)\). By the implicit function theorem, \(t(\zeta_1,\zeta_2)\) is smooth on \(\{|\zeta_1\zeta_2|<1/2\}\).
Substituting this function into \eqref{tag-4} extends \(u_3\) smoothly across the coordinate axes.

Thus \(u_3\) is continuous on \(\mathbb C^2\) and smooth outside
\[
\mathcal S
=\left\{(\zeta_1,\zeta_2)\in\mathbb C^2\,\middle|\,\zeta_1\zeta_2=1\right\}.
\]
On its original domain \(\{\zeta_1\zeta_2\notin E\cup\{0\}\}\), the complex Hessian of \(u_3\) is positive definite with determinant one. This domain is dense in \(\mathbb C^2\setminus\mathcal S\). By continuity of the complex Hessian, it is therefore positive semidefinite with determinant one throughout \(\mathbb C^2\setminus\mathcal S\), and hence is positive definite there. Consequently, \(u_3\) is strictly plurisubharmonic and satisfies the Monge–Ampère equation on \(\mathbb C^2\setminus\mathcal S\).
The set \(\mathcal S\) is closed and is the \(-\infty\) locus of the plurisubharmonic function \(\log|\zeta_1\zeta_2-1|\), so it is complete pluripolar. The pluripolar extension lemma therefore gives
\begin{equation}\label{prw21:eq:global-equation}
u_3\in PSH(\mathbb C^2)\cap C(\mathbb C^2),
\qquad
(dd^cu_3)^2=\beta_2
\end{equation}
on \(\mathbb C^2\) in the Bedford–Taylor sense. Then by Proposition \ref{prop:bt-viscosity}, \(u_3\) is also a viscosity solution.

\subsection{Properness}

We prove that there exists a constant $K>0$ such that
\begin{equation}\label{prw21:eq:proper-growth}
u_3(\zeta_1,\zeta_2)
\geq2\log\lp1+|\zeta_1|^2+|\zeta_2|^2\rp-K
\end{equation}
on $\bC^2$.

By definition, the supremum defining $U_3$ is taken over
$t\in\bR$ when $z_1\notin E\cup\{0\}$, and over $t\leq1$
when $z_1\in E$. In both cases, $t=0,\pm1$ are admissible.
Thus, for every $z_1\ne0$ and $x_2\in\bR$, we have
$$
U_3(z_1,x_2)
\geq4\max_{t\in\{-1,0,1\}}
\left\{x_2t-v_3(z_1,t)\right\}.
$$
We estimate the three test values separately.

For each fixed $q\geq0$, the function $h_W(z_1,q)$ is
nonincreasing in $|z_1|$. Hence
$$
h_W(z_1,q)\leq h_W(0,q)
=\frac12q\lp\log(2q)-1\rp,
$$
with the convention $0\log0=0$.
For $|t|\leq1$, the quantities $1+2^{-j}-t$ and $1-t$
lie in $[0,5/2]$. Since $h_W(0,q)$ is bounded above on
this interval and the weights in $C$ and $S$ are positive
and summable, there exists a constant $C_0>0$, independent
of $z_1$ and $t$, such that
$$
4(C+S)(z_1,t)\leq C_0,
\qquad z_1\in\bC,\quad |t|\leq1.
$$
Moreover,
$$
\begin{aligned}
h_W(z_1,0)=-\frac12|z_1|,\qquad h_W(z_1,1)\leq\frac12(\log2-1)<0.
\end{aligned}
$$

For $\zeta_1\zeta_2\ne0$, taking $t=0$ gives
$$
\begin{aligned}
u_3(\zeta_1,\zeta_2)
&\geq-4v_3(\zeta_1\zeta_2,0)\\
&=2|\zeta_1\zeta_2|-4(C+S)(\zeta_1\zeta_2,0)\\
&\geq-C_0.
\end{aligned}
$$
Taking $t=1$, we similarly obtain
\begin{equation*}
\begin{aligned}
u_3(\zeta_1,\zeta_2)
&\geq4\log|\zeta_2|-4v_3(\zeta_1\zeta_2,1)\\
&\geq4\log|\zeta_2|-C_0.
\end{aligned}
\end{equation*}
Finally, for $z_1\ne0$, we have
$$
h_W(z_1,-1)=h_W(z_1,1)-\log|z_1|.
$$
Taking $t=-1$ and using this identity gives
\begin{equation*}
\begin{aligned}
u_3(\zeta_1,\zeta_2)
&\geq-4\log|\zeta_2|-4v_3(\zeta_1\zeta_2,-1)\\
&=4\log|\zeta_1|-4h_W(\zeta_1\zeta_2,1)
-4(C+S)(\zeta_1\zeta_2,-1)\\
&\geq4\log|\zeta_1|-C_0.
\end{aligned}
\end{equation*}

Combining these three estimates, we obtain
$$
u_3(\zeta_1,\zeta_2)
\geq4\log\max\{1,|\zeta_1|,|\zeta_2|\}-C_0.
$$
Since
$$
1+|\zeta_1|^2+|\zeta_2|^2
\leq3\max\{1,|\zeta_1|,|\zeta_2|\}^2,
$$
this proves \eqref{prw21:eq:proper-growth} with
$K=C_0+2\log3$ whenever $\zeta_1\zeta_2\ne0$.
By continuity of $u_3$, the same estimate holds on the
coordinate axes, and hence throughout $\bC^2$.

The growth estimate implies that
$u_3(\zeta_1,\zeta_2)\to+\infty$ as
$|\zeta_1|^2+|\zeta_2|^2\to+\infty$.
Thus every sublevel set of $u_3$ is bounded.
Since $u_3$ is continuous, its closed sublevel sets are compact.
Therefore $u_3$ is a continuous proper exhaustion of $\bC^2$.

\subsection{Failure of $W^{2,1}_{\mathrm{loc}}$, $W^{1,p}_{\mathrm{loc}}$ ($p>2$) regularity and Dini continuity}

\subsubsection{Failure of $W^{2,1}_{\mathrm{loc}}$ regularity}
\label{prw21:step:regularity}
Let $D_k=B_{2^{-k-3}}(\xi_k)$ and $I=[x_*+1,x_*+2]$,
where $x_*$ is defined by \eqref{endpoint slope}.
We claim that there exists a constant $a>0$ such that, for all sufficiently
large $k$,
\begin{equation}\label{prw21:eq:boundary-gap}
U_3(z_1,x_2)-U_3(\xi_k,x_2)\geq\frac ak,
\qquad z_1\in\p D_k,\quad x_2\in I.
\end{equation}
In the following estimates, $c$ and $C$ denote positive constants
independent of $a,j,k,z_1,x_2$, and may change from line to line.

To prove the claim, we first estimate $(v_3)_t$ on $\p D_k$ at $t=1+a/k$. We then use the definition of $U_3$ and integrate $x_2-(v_3)_t(z_1,t)$ from $1$ to $1+a/k$ to estimate $U_3(z_1,x_2)$ from below.
By the choice of the centers and radii of $D_k$, for $z_1\in\p D_k$ we have
$$
c2^{-k}\leq|z_1-1|\leq C 2^{-k},
\qquad
|z_1-\xi_j|\geq2^{-j-3}\quad(j\geq4).
$$
Thus \eqref{prw21:eq:total-harmonic} gives
$$
\begin{aligned}
(v_3)_{tt}(z_1,t)
&=\frac1{2\sqrt{|z_1|^2+t^2}}
+\sum_{j\geq1}\frac{2^{-j}}{\sqrt{|z_1-1|^2+(s_j-t)^2}}
+\sum_{j\geq4}\frac{3^{-j}}{\sqrt{|z_1-\xi_j|^2+(1-t)^2}}\\
&\leq\frac12+\sum_{j\geq1}\frac{2^{-j}}{s_j-t}
+8\sum_{j\geq4}\lp \frac 23\rp^j ,
\qquad z_1\in\p D_k,\quad 0\leq t<1.
\end{aligned}
$$
The right-hand side
is integrable on $[0,1]$ and independent of $k$ and $z_1$.
Recall $x_*=\lim_{t\to 1^-}(v_3)_t(1,t)$ in \eqref{endpoint slope}, we have
$$
\begin{aligned}
(v_3)_t\lp z_1,1+\frac ak \rp -x_*
=\,&(v_3)_t(z_1,0)-(v_3)_t(1,0)+\int_0^1
\lp (v_3)_{tt}(z_1,t)-(v_3)_{tt}(1,t)\rp\,dt\\
&+\int_1^{1+a/k}(v_3)_{tt}(z_1,t)\,dt.
\end{aligned}
$$
The first two terms are $o(1)$ by continuity and dominated
convergence, uniformly for $z_1\in\p D_k$.
It suffices to estimate the third term $\int_1^{1+a/k}(v_3)_{tt}(z_1,t)\,dt$. Note that $v_3=h_W+S+C$, and $\lp h_W+S\rp_{tt}(z_1,1+d)$ is uniformly bounded for $z_1\in \p D_k$ and $0\le d\le \frac14$. We need to study the contribution of $C$. 
By direct computations, we have
\begin{equation*}
\begin{aligned}
0\leq C_t\lp z_1,1+\frac ak\rp-C_t(z_1,1)
&\le \sum_j2^{-j}\int_0^{a/k}
\frac{ds}{\sqrt{|z_1-1|^2+(s-2^{-j})^2}}\\
&\leq C\frac ak\log\frac ka
+\frac{4a}{k}\int_{-2a/k}^{2a/k}
\frac{ds}{\sqrt{|z_1-1|^2+s^2}}\\
&\leq C\frac ak\log\frac1{|z_1-1|}\\
&\leq Ca.
\end{aligned}
\end{equation*}
Here we used that there are at most $C\log(k/a)$ indices with
$2^{-j}\geq2a/k$, and each corresponding term is at most
$2a/k$. 
We also used $\log(k/a)\leq\log(1/|z_1-1|)$ and
$$
\int_{-2a/k}^{2a/k}\frac{ds}{\sqrt{|z_1-1|^2+s^2}}
=2\log
\frac{2a/k+\sqrt{4a^2/k^2+|z_1-1|^2}}{|z_1-1|}
\leq C\log\frac1{|z_1-1|}.
$$

Thus
\begin{equation}\label{prw21:eq:boundary-slope}
(v_3)_t\left(z_1,1+\frac{a}{k}\right)
\leq x_*+Ca+o(1)
\leq x_*+\frac12,
\qquad z_1\in\p D_k,
\end{equation}
where we first choose $a>0$ sufficiently small and then
take $k$ sufficiently large.

Now, we can estimate the lower bound of $U_3(z_1,x_2)$ on $\p D_k\times I$. Since $(v_3)_t(z_1,\cdot)$ is nondecreasing, \eqref{prw21:eq:boundary-slope} implies
$$
x_2-(v_3)_t(z_1,t)\geq\frac12,
\qquad
z_1\in\p D_k,\quad x_2\in I,\quad
1\leq t\leq1+a/k.
$$
Taking $t=1+a/k$ in
\eqref{prw21:eq:cylindrical-potential}, we therefore have
$$
\begin{aligned}
U_3(z_1,x_2)
&\geq4\lp x_2(1+a/k)-v_3(z_1,1+a/k)\rp\\
&=4\lp x_2-v_3(z_1,1)\rp
+4\int_1^{1+a/k}
\lp x_2-(v_3)_t(z_1,t)\rp\,dt\\
&\geq4\lp x_2-v_3(z_1,1)\rp+\frac{2a}{k}.
\end{aligned}
$$

We next estimate $U_3$ at $\xi_k$.
The $k$-th term of $S$ contributes
$-3^{-k}\log(2(1-t))$ to $(v_3)_t(\xi_k,t)$,
which tends to $+\infty$ as $t\uparrow1$.
Using monotonicity
for the other terms, we obtain, for $t<1$,
$$
\begin{aligned}
(v_3)_t(\xi_k,t)+3^{-k}\log(2(1-t))
\leq & \frac12\log\left(\sqrt{|\xi_k|^2+1}+1\right)
-\sum_{j=1}^{\infty}2^{-j}
\log\left(\sqrt{2^{-2k}+2^{-2j}}+2^{-j}\right)\\
&\qquad\quad-\sum_{\substack{j\geq4\\j\ne k}}
3^{-j}\log|\xi_k-\xi_j|
=x_*+o(1).
\end{aligned}
$$
Here $|\xi_k-\xi_j|\geq2^{-j-1}$ for $j\ne k$.
Since $x_2\geq x_*+1$ on $I$, for all sufficiently large $k$,
$$
x_2-(v_3)_t(\xi_k,t)
\geq3^{-k}\log(2(1-t)),
\qquad x_2\in I,\quad t<1.
$$
Integrating from $t$ to $1$, we obtain
$$
\begin{aligned}
x_2t-v_3(\xi_k,t)-\lp x_2-v_3(\xi_k,1)\rp
&=-\int_t^1\lp x_2-(v_3)_t(\xi_k,s)\rp\,ds\\
&\leq-3^{-k}\int_{1/2}^1\log(2(1-s))\,ds\\
&=\frac{3^{-k}}2.
\end{aligned}
$$
The logarithmic singularity at $s=1$ is integrable.
Taking the supremum over $t$ and using
\eqref{prw21:eq:capped-conjugate} gives
$$
U_3(\xi_k,x_2)
\leq4\lp x_2-v_3(\xi_k,1)\rp+2\cdot3^{-k},
\qquad x_2\in I.
$$

Finally, $2h_W(\cdot,q)$ is $1$-Lipschitz for $q\geq0$.
Since the weights are summable, $v_3(\cdot,1)$ is Lipschitz,
and hence
$$
|v_3(z_1,1)-v_3(\xi_k,1)|
\leq C|z_1-\xi_k|
\leq C 2^{-k},
\qquad z_1\in\p D_k.
$$
Subtracting the center estimate from the boundary estimate gives
$$
\begin{aligned}
U_3(z_1,x_2)-U_3(\xi_k,x_2)
&\geq\frac{2a}{k}
-4\lp v_3(z_1,1)-v_3(\xi_k,1)\rp-2\cdot3^{-k}\\
&\geq\frac{2a}{k}-C 2^{-k}-2\cdot3^{-k}\\
&\geq\frac ak.
\end{aligned}
$$
This proves \eqref{prw21:eq:boundary-gap}.

By Section~\ref{prw21:step:pluripotential}, $U_3(\cdot,x_2)$ is smooth on
$\overline D_k$ for $x_2\in I$.
Applying Proposition~\ref{prop:real-hessian-bound}, using
\eqref{prw21:eq:boundary-gap}, and integrating over $I$, we obtain
\begin{equation}\label{prw21:eq:harmonic-series}
\int_I\int_{D_k}|\D^2_{z_1}U_3|\,dV(z_1)\,dx_2\geq\frac ck.
\end{equation}

Fix a closed interval $J\subset\bR$ with $0<|J|<2\pi$, and set
$$
B_k=D_k\times\left\{x_2+\sq y_2\,\middle|\,x_2\in I,\ y_2\in J\right\}.
$$
Recall from Section~\ref{partial Legendre} that the map
$\Psi(z_1,z_2)=(z_1e^{-z_2},e^{z_2})$ locally inverts $\Pi$.
It is injective on the strip
$\bC\times(\bR+\sq J)$ and has real Jacobian determinant $1$.
For fixed $\zeta_2$, \eqref{prw21:eq:entire-pullback} gives
$$
|\D^2_{\zeta_1}u_3|\circ\Psi
=e^{2x_2}|\D^2_{z_1}U_3|,
$$
where $\D^2_{\zeta_1}$ is the real Hessian in $\zeta_1$.
Using \eqref{prw21:eq:harmonic-series}, integrating over $J$,
and changing variables by $\Psi$, we obtain
$$
\int_{\Psi(B_k)}\left|\D^2u_3\right|\,dV\geq\frac ck.
$$
The sets $\Psi(B_k)$ are pairwise disjoint and lie in a fixed
compact subset of $\bC^2$.
Suppose that $u_3\in W^{2,1}_{\mathrm{loc}}(\bC^2)$.
Its weak Hessian then agrees with its classical Hessian
on these sets.
Summing the preceding lower bounds contradicts the local
integrability of $\D^2u_3$.
Therefore $u_3\notin W^{2,1}_{\mathrm{loc}}(\bC^2)$.

\subsubsection{Failure of Dini continuity}
Set $x_0=x_*+1$ and
$q_0=\Psi(1,x_0)=(e^{-x_0},e^{x_0})$.
Define
$$
\omega(\rho)=\sup_{|\zeta-q_0|\leq\rho}
|u_3(\zeta)-u_3(q_0)|.
$$
Both $\Psi(\xi_k+2^{-k-3},x_0)$ and $\Psi(\xi_k,x_0)$
lie within distance $C_0 2^{-k}$ of $q_0$,
where $C_0>0$ is independent of $k$.
Taking $z_1=\xi_k+2^{-k-3}$ and $x_2=x_0$ in
\eqref{prw21:eq:boundary-gap} and using the triangle inequality,
we obtain
$$
2\omega(C_0 2^{-k})\geq\frac ak
$$
for all sufficiently large $k$.

For each sufficiently small $\rho>0$, choose $k$ such that
$C_0 2^{-k}\leq\rho<C_0 2^{-k+1}$.
Since $\omega$ is nondecreasing, we have
\begin{equation}\label{prw21:eq:nondini-modulus}
\omega(\rho)\geq\frac{a}{2k}
\geq\frac c{\log(1/\rho)},
\end{equation}
where $c>0$ is independent of $\rho$.
Hence $\int_0^{1/2}\omega(\rho)\,d\rho/\rho=+\infty$,
so $u_3$ is not Dini continuous at $q_0$.

\subsubsection{Failure of $W^{1,p}_{loc}$ for any $p>2$.}
Fix $p>2$ and suppose, to the contrary, that
$u_3\in W^{1,p}_{\mathrm{loc}}(\bC^2)$. We use a similar argument to show the contradiction. 
Recall that $D_k=B_{r_k}(\xi_k)$, $r_k=2^{-k-3}$,
$I=[x_*+1,x_*+2]$. 
By \eqref{prw21:eq:boundary-gap}, there exists $a>0$ such that
$$
U_3(z_1,x_2)-U_3(\xi_k,x_2)\geq\frac ak,
\qquad z_1\in\p D_k,\quad x_2\in I,
$$
for all sufficiently large $k$.
Since $U_3(\cdot,x_2)$ is smooth near $\ol{D_k}$, integration
along the radius and H\"older's inequality give, for every
$\theta\in[0,2\pi]$,
$$
\begin{aligned}
\frac ak
&\leq\int_0^{r_k}
|\D_{z_1}U_3(\xi_k+re^{\sq\theta},x_2)|\,dr\\
&\leq
\left(\int_0^{r_k}
|\D_{z_1}U_3(\xi_k+re^{\sq\theta},x_2)|^p r\,dr\right)^{1/p}
\left(\int_0^{r_k}r^{-1/(p-1)}\,dr\right)^{(p-1)/p}.
\end{aligned}
$$
Since $p>2$,
$$
\int_0^{r_k}r^{-1/(p-1)}\,dr
=\frac{p-1}{p-2}r_k^{(p-2)/(p-1)}.
$$
Taking the power $p$ on the both sides and we obtain
$$
\int_{D_k}|\D_{z_1}U_3(z_1,x_2)|^p\,dV(z_1)
\geq
2\pi\left(\frac{p-2}{p-1}\right)^{p-1}
\frac{a^p}{k^p}r_k^{2-p},
\qquad x_2\in I.
$$

We use the interval \(J\) and the sets \(B_k\) introduced above. For fixed $z_2=x_2+\sq y_2$, the identity
$$
U_3(z_1,x_2)=u_3\bigl(z_1e^{-z_2},e^{z_2}\bigr)
$$
gives
$$
|\D_{\zeta_1}u_3|\circ\Psi
=e^{x_2}|\D_{z_1}U_3|.
$$
Arguing as in the preceding proof of failure of local $W^{2,1}$
regularity, we obtain
$$
\int_{\Psi(B_k)}|\D u_3|^p\,dV
\geq c_p\frac{2^{(p-2)(k+3)}}{k^p},
$$
where $c_p>0$ is independent of $k$.
Since $\Psi(B_k)$ lie in a fixed compact set and the lower bound
diverges for $p>2$, we have
$u_3\notin W^{1,p}_{\mathrm{loc}}(\bC^2)$ for every $p>2$.

\subsection{Conclusion}

By \eqref{prw21:eq:global-equation} and
Proposition~\ref{prop:bt-viscosity}, $u_3$ solves
\eqref{CMAdensity} in both the Bedford--Taylor and viscosity senses.
Together with the properness and the three failures of regularity
proved above, this proves Theorem~\ref{theo:proper-w21-main}.

To prove Corollary~\ref{cor:bounded-sublevel}, choose a closed ball $K$
whose interior contains $q_0$ and all the sets $\Psi(B_k)$ used above,
fix $c>\max_Ku_3$, and let $\Om_*$ be the connected component
of $\{u_3<c\}$ containing $K$.
By properness and continuity, $\Om_*$ is bounded and hyperconvex,
and $v=u_3-c\in PSH(\Om_*)\cap C(\ol{\Om_*})$ solves
\eqref{CMAdensity} in both senses with zero boundary values,
while \eqref{prw21:eq:nondini-modulus} and
\eqref{prw21:eq:harmonic-series} give the asserted failures of regularity since $K\Subset\Om_*$.

\section{A continuous solution with a singular Laplacian}
\label{sec:cantor-laplacian}
In this section, we prove Theorem~\ref{theo:singular-laplacian-main}.
Our aim is to introduce a singular part in the Laplacian
while preserving the complex Monge--Amp\`ere equation.
We first define $g$ as the potential function of a Cantor probability
measure $\nu$ supported on a Cantor set $E$.
We then construct a continuous plurisubharmonic function $W$ such that
$u_4=W+g$ solves the complex Monge--Amp\`ere equation with density $1$
in the Bedford--Taylor sense.
Since $\Delta_{z_1}g=2\pi\nu$ and $W$ is plurisubharmonic,
the distributional Laplacian of $u_4$ has a singular part. 

\subsection{Construction}
Our construction is motivated by the following idea.
Suppose that $W$ is plurisubharmonic on a domain $\Om\subset\bC^2$,
and let $g$ be a subharmonic function depending only on $z_1$.
A direct computation gives
$$
\lp dd^c(W+g)\rp^2
=(dd^cW)^2
+2d_{z_2}d_{z_2}^cW\wedge d_{z_1}d_{z_1}^cg,
$$
and
$$
\Delta(W+g)=\Delta W+\Delta_{z_1}g
\ge\Delta_{z_1}g
$$
in the sense of distributions.
This suggests constructing $W$ and $g$ so that
$(dd^cW)^2=\beta_2$ and the mixed term vanishes.
So adding $g$ preserves the complex Monge--Amp\`ere measure of $W$.
On the other hand, if $\Delta_{z_1}g$ is a nonzero positive singular
measure, then the Laplacian of $u=W+g$ also has a nonzero singular part.

First, we explain the choice of $g$. A natural choice is the logarithmic potential of a
compactly supported finite positive singular measure $\nu$ on $\bC$:
\begin{align}\label{defi-g}
g(z_1)=\int_{\bC}\log|z_1-a|\,d\nu(a),
\qquad z_1\in\bC.
\end{align}
As recalled in Section~\ref{pre}, $g$ is subharmonic and satisfies
$\Delta_{z_1}g=2\pi\nu$ in the sense of distributions.

To ensure that $u=W+g$ is continuous, we construct $W$ to be
continuous and choose $\nu$ so that its logarithmic potential
is continuous.
By the logarithmic potential estimate in Section~\ref{pre},
it suffices to require that
$$
\nu(B_r(z_1))\le Cr^d,
\qquad z_1\in\bC,\quad 0<r\le1,
$$
for some $C>0$ and $d>0$.

For this purpose, we take \(\nu\) to be the Cantor probability measure on \(E=(2C-1)/4\) introduced in Section ~\ref{pre}, where \(C\) is the middle-third Cantor set. 
This is the normalized $d$-dimensional Hausdorff measure on $E$.
Then $\operatorname{supp}\nu=E$.
Since $E$ has planar Lebesgue measure zero, $\nu$ is singular with
respect to planar Lebesgue measure.

With this choice of $\nu$, the potential $g$ in \eqref{defi-g} takes the form
$$
g(z_1)=\int_E\log|z_1-a|\,d\nu(a).
$$
The logarithmic estimate \eqref{log-est} implies that $g$ is continuous.
As recalled in Section~\ref{pre}, $g$ is harmonic on $\bC\setminus E$
and satisfies $\Delta_{z_1}g=2\pi\nu$ in the sense of distributions.

We now explain the choice of $W$. We use the partial Legendre transformation in
$x_2=\text{Re}\,z_2$.
At corresponding smooth points, the relations
$$
x_2=v_t(z_1,t),\qquad
W(z_1,x_2)=4\lp x_2t-v(z_1,t)\rp
$$
give
$$
W_{z_2\bar z_2}=\frac1{v_{tt}},\qquad
\det\lp W_{i\bar j}\rp
=-\frac{\Delta_{z_1}v}{v_{tt}}.
$$
Thus, if $v_{tt}>0$ and
$\Delta_{z_1}v+v_{tt}=0$, then $W$ is plurisubharmonic
and $\det(W_{i\bar j})=1$.
Moreover, for $u=W+g$, we have
$$
\det\lp u_{i\bar j}\rp
=\frac{-\Delta_{z_1}v+\frac14\Delta_{z_1}g}{v_{tt}}.
$$
Since $g$ is harmonic outside $E$, this suggests choosing $v$
to be harmonic on $(\bC\setminus E)\times\bR$
and strictly convex in $t$ on a suitable region.

For this purpose, we introduce a nonnegative even weight $\rho$
with $\rho(0)=0$ and consider the Poisson equation
$$
\Delta_{z_1}v+v_{tt}
=-4\pi\nu\otimes
\lp\chi_{[-L,L]}(t)\rho(t)\,dt\rp
$$
in the sense of distributions on $\bC\times\bR$.
The corresponding Newton potential is
$$
v(z_1,t)=\int_E\int_{-L}^L
\frac{\rho(s)}{\sqrt{|z_1-a|^2+(t-s)^2}}
\,ds\,d\nu(a).
$$
The increase of $\rho$ away from $0$ gives a positive contribution
to $(v_\eps)_{tt}$ near $t=0$, where $v_\eps$ are the smooth
regularizations constructed below. Hence $(v_{\eps})_{tt}$ has positive lower bound uniformly with respect to $\eps\to 0$. 
With a suitable rate of vanishing of $\rho$ and $L$ sufficiently
large, we will obtain $t_{\eps}$ by 
$x_2=(v_\eps)_t(z_1,t_\eps)$ on a fixed interval about $x_2=0$,
along which $(v_\eps)_{tt}$ tends uniformly to infinity as
$\operatorname{dist}(z_1,E)
+\eps\to0$.

Let $g_\eps$ be the standard mollifications of $g$.
Since $\Delta_{z_1}g_\eps$ is supported in the $\eps$-neighborhood
of $E$ and has uniformly bounded mass, the preceding lower bound
will imply
$$
(W_{\eps})_{z_2\bar z_2}(g_{\eps})_{z_1\bar z_1}=\frac{\Delta_{z_1}g_\eps(z_1)}
     {4(v_\eps)_{tt}(z_1,t_\eps(z_1,x_2))}
\to0
\qquad\text{in }L^1_{\mathrm{loc}}
\qquad\text{as }\eps\to0.
$$
This gives the desired convergence of the mixed term.
As the estimates in the next subsection will show,
a suitable choice is $\rho(t)=|t|^\alpha$, $0<\alpha<1-d$. 

Now, we start to explain the explicit construction of $u_4$. Set \(\rho(s)=|s|^{1/4}\) and choose \(L>0\) sufficiently large, as specified in the next subsection.
The corresponding superharmonic Newton potential is
\begin{equation*}
v(z_1,t)=\int_{-L}^{L}\int_E
\frac{|s|^{1/4}}{\sqrt{|z_1-a|^2+(t-s)^2}}
\,d\nu(a)\,ds.
\end{equation*}
In the next subsection, we will prove that $v$ is bounded and
continuous on $\bC\times\bR$ and smooth on
$(\bC\setminus E)\times\bR$.

We then define $W$ by the normalized partial Legendre transformation.
For $z_1\in \bD\setminus E$, let $t=t(z_1,x_2)$ be the function
through $t=0$ determined by
$$
x_2=v_t(z_1,t),\qquad |x_2|<R,
$$
where $R>0$ will be chosen to make $t$ well-defined in the next subsection, 
and set
\begin{equation*}
W(z_1,x_2)=4\lp x_2t-v(z_1,t)\rp,
\qquad z_1\in \bD\setminus E.
\end{equation*}
On $E$, define
$$
W(\xi,x_2)=-4v(\xi,0),
\qquad \xi\in E,\quad |x_2|<R.
$$
We will prove in the next subsection that this extension is continuous. 

Finally, set
\begin{equation*}
\Om=\bD\times\{z_2\in\bC\,|\,|\text{Re}\,z_2|<R\},
\end{equation*}
and define
\begin{equation*}
u_4(z_1,z_2)=W(z_1,\text{Re}\,z_2)+g(z_1).
\end{equation*}
Then $u_4$ is independent of $\text{Im}\,z_2$.
For $z_1\in E$, it is independent of $z_2$.

\subsection{Continuity}
We first verify that $v$ is bounded and continuous on $\bC\times\bR$.
By definition, the measure $d\nu(a)|s|^{1/4}\,ds$ on $E\times[-L,L]$
is finite. To estimate $v(z_1,t)$, we first integrate in $s$. The function $s\to(|z_1-a|^2+s^2)^{-1/2}$ is even and
decreases as $|s|$ increases. Hence, by $|s|^{1/4}\le L^{1/4}$, 
$$
\begin{aligned}
v(z_1,t)
&\le L^{1/4}\int_E\int_{-L}^{L}
\frac{1}{\sqrt{|z_1-a|^2+(t-s)^2}}
\,ds\,d\nu(a)\\
&\le L^{1/4}\int_E\int_{-L}^{L}
\frac{1}{\sqrt{|z_1-a|^2+s^2}}
\,ds\,d\nu(a)\\
&=2L^{1/4}\int_E
\log\frac{L+\sqrt{L^2+|z_1-a|^2}}{|z_1-a|}
\,d\nu(a)\\
&\le 2L^{1/4}\log\bigl(L+\sqrt{L^2+1}\bigr)
+2L^{1/4}\int_E\log^+\frac1{|z_1-a|}\,d\nu(a).
\end{aligned}
$$
Here $\log^+x=\max\{\log x,0\}$. 
By the Cantor mass estimate \eqref{cant:integral-measure-growth} and the layer-cake formula,
we have
$$
\begin{aligned}
\int_E\log^+\frac1{|z_1-a|}\,d\nu(a)
&=\int_0^\infty
\nu\left(\left\{a\in E\,\middle|\,
\log^+\frac1{|z_1-a|}>t\right\}\right)\,dt\\
&=\int_0^\infty\nu(B_{e^{-t}}(z_1))\,dt\\
&\le C\int_0^\infty e^{-dt}\,dt
=\frac Cd,
\qquad z_1\in\bC.
\end{aligned}
$$
To show the continuity, we let $0<\eps\le1/4$, define
$$
v_\eps(z_1,t)=\int_{-L}^{L}\int_E
\frac{|s|^{1/4}\,d\nu(a)\,ds}
{\sqrt{|z_1-a|^2+(t-s)^2+\eps^2}}
.
$$
Then $v_\eps$ is smooth on $\bC\times\bR$. 
By direct computations, 
$$
\begin{aligned}
0\le v(z_1,t)-v_\eps(z_1,t)
&\le L^{1/4}\int_E\int_{\bR}
\left(
\frac1{\sqrt{|z_1-a|^2+(t-s)^2}}
-\frac1{\sqrt{|z_1-a|^2+(t-s)^2+\eps^2}}
\right)
\,ds\,d\nu(a)\\
&=2L^{1/4}\int_E
\left[
\log\frac{s+\sqrt{|z_1-a|^2+s^2}}
{s+\sqrt{|z_1-a|^2+s^2+\eps^2}}
\right]_{s=0}^{s=\infty}
\,d\nu(a)\\
&=L^{1/4}\int_E
\log\left(1+\frac{\eps^2}{|z_1-a|^2}\right)
\,d\nu(a)\\
&=2L^{1/4}\eps^2\int_0^\infty
\frac{\nu(B_r(z_1))}{r(r^2+\eps^2)}\,dr\\
&\le C\eps^d.
\end{aligned}
$$
This indicates that $v_\eps\to v$ uniformly on $\bC\times\bR$,
hence $v$ is continuous. Moreover, on every compact subset of
$(\bC\setminus E)\times\bR$, $|z_1-a|$ is bounded away
from zero for $a\in E$. Differentiating under the integral 
shows that $v$ is smooth and harmonic on $(\bC\setminus E)\times\bR$.

Next, we choose $R>0$ and construct 
$t_\eps$ and $t$ corresponding to $v_\eps$ and $v$, respectively.
We also prove that $t$ extends continuously across $E$.
Let $z_1\in\bD$ and $0<\eps\le1/4$.
Differentiating under the integral, we obtain
$$
(v_\eps)_{tt}(z_1,t)
=\int_E\int_{-L}^L
\frac{\lp2(t-s)^2-|z_1-a|^2-\eps^2\rp|s|^{1/4}}
     {\lp|z_1-a|^2+(t-s)^2+\eps^2\rp^{5/2}}
\,ds\,d\nu(a).
$$
To estimate this integral, we consider the one-variable function
$$
H_\ell(q)=\int_{-\ell}^{\ell}
\frac{|s|^{1/4}}{\sqrt{1+(q-s)^2}}\,ds,
\qquad \ell>0.
$$
By scaling, we can write
$$
v_\eps(z_1,t)
=\int_E
\lp|z_1-a|^2+\eps^2\rp^{1/8}
H_{\frac{L}{\sqrt{|z_1-a|^2+\eps^2}}}
\lp\frac{t}{\sqrt{|z_1-a|^2+\eps^2}}\rp
\,d\nu(a),
$$
and 
$$
(v_\eps)_{tt}(z_1,t)
=\int_E
\lp|z_1-a|^2+\eps^2\rp^{-7/8}
H_{\frac{L}{\sqrt{|z_1-a|^2+\eps^2}}}''
\lp\frac{t}{\sqrt{|z_1-a|^2+\eps^2}}\rp
\,d\nu(a).
$$

Direct computation gives
$$
H_\ell'''(q)
=3\int_{-\ell}^{\ell}
\frac{|s|^{1/4}(q-s)\lp3-2(q-s)^2\rp}
     {\lp1+(q-s)^2\rp^{7/2}}\,ds.
$$
Hence, for $|q|\le1$,
$$
|H_\ell'''(q)|
\le C\int_{\bR}\frac{|s|^{1/4}}{(1+|s|)^4}\,ds
\le C,
$$
where $C$ is independent of $\ell$.
On the other hand, integration by parts gives
$$
H_\ell''(0)
=\frac12\int_0^\ell
\frac{s^{1/4}}{(1+s^2)^{3/2}}\,ds
-\frac{2\ell^{5/4}}{(1+\ell^2)^{3/2}}.
$$
Thus
$$
H_\ell''(0)\longrightarrow
\frac12\int_0^\infty
\frac{s^{1/4}}{(1+s^2)^{3/2}}\,ds>0
\qquad\text{as }\ell\to\infty.
$$
Choose $\ell_0>0$ and $c>0$ so that
$H_\ell''(0)\ge2c$ for $\ell\ge\ell_0$.
Taking $\eta\in(0,1]$ sufficiently small that $C\eta\le c$,
the mean value theorem yields
$$
H_\ell''(q)\ge H_\ell''(0)-C|q|\ge c,
\qquad \ell\ge\ell_0,\quad |q|\le\eta.
$$

We now fix $L\ge 2\ell_0$, and \(
\kappa=\min\left\{\frac{\eta}{\sqrt2},\frac18\right\}.
\) 
Since $E\subset \overline{B_{\frac14}(0)}\subset \bD$, we have
$$
\operatorname{dist}(z_1,E)+\eps\le\frac32,
\qquad
\sqrt{|z_1-a|^2+\eps^2}\le2,
\qquad a\in E.
$$
Moreover,
$$
\operatorname{dist}(z_1,E)+\eps
\le |z_1-a|+\eps
\le\sqrt2\sqrt{|z_1-a|^2+\eps^2}.
$$
Consequently, if
$$
|t|\le\kappa\lp\operatorname{dist}(z_1,E)+\eps\rp,
$$
then
$$
\frac{L}{\sqrt{|z_1-a|^2+\eps^2}}\ge\ell_0,
\qquad
\frac{|t|}{\sqrt{|z_1-a|^2+\eps^2}}
\le\sqrt2\,\kappa\le\eta.
$$
Applying the lower bound for $H_\ell''$ in the scaled integral,
we obtain
$$
(v_\eps)_{tt}(z_1,t)
\ge c\int_E
\lp|z_1-a|^2+\eps^2\rp^{-7/8}\,d\nu(a).
$$

Choose $a_0\in E$ such that
$|z_1-a_0|=\operatorname{dist}(z_1,E)$.
For $a\in E\cap B_{\operatorname{dist}(z_1,E)+\eps}(a_0)$,
the triangle inequality gives
$$
\sqrt{|z_1-a|^2+\eps^2}
\le |z_1-a|+\eps
\le2\lp\operatorname{dist}(z_1,E)+\eps\rp.
$$
The lower mass bound for $\nu$ therefore yields
$$
\begin{aligned}
(v_\eps)_{tt}(z_1,t)
&\ge c\lp\operatorname{dist}(z_1,E)+\eps\rp^{-7/4}
\nu\lp B_{\operatorname{dist}(z_1,E)+\eps}(a_0)\rp\\
&\ge c\lp\operatorname{dist}(z_1,E)+\eps\rp^{d-7/4}
\ge c>0.
\end{aligned}
$$
Here the constants are independent of $z_1$, $t$, and $\eps$.
The mass lower bound remains valid for radii up to $3/2$
after decreasing its constant, since
$\nu(B_r(a_0))=1$ for $r\ge1$.

Since $(v_\eps)_t(z_1,0)=0$, it follows that
$$
\begin{aligned}
(v_\eps)_t
\lp z_1,\kappa\lp\operatorname{dist}(z_1,E)+\eps\rp\rp=
&\int_0^{\kappa(\operatorname{dist}(z_1,E)+\eps)}
(v_\eps)_{tt}(z_1,s)\,ds\\
\ge&
c\kappa\lp\operatorname{dist}(z_1,E)+\eps\rp^{d-3/4}
\ge c_0>0,
\end{aligned}
$$
where $c_0$ is independent of $z_1$ and $\eps$.
The last inequality follows from $d<3/4$ and
$\operatorname{dist}(z_1,E)+\eps\le3/2$.

Fix $0<R<c_0$.
Since $(v_\eps)_t$ is odd and strictly increasing for
$|t|\le\kappa\lp\operatorname{dist}(z_1,E)+\eps\rp$,
the equation
$$
x_2=(v_\eps)_t(z_1,t)
$$
has a unique solution in this interval for every $|x_2|<R$.
Denote this solution by $t_\eps(z_1,x_2)$.
The implicit function theorem shows that $t_\eps$ is smooth
on $\bD\times(-R,R)$, and
\begin{equation}\label{teps-bound}
|t_\eps(z_1,x_2)|
\le\kappa\lp\operatorname{dist}(z_1,E)+\eps\rp
\le\frac{\operatorname{dist}(z_1,E)+\eps}{8}.
\end{equation}

The same argument applies to $\eps=0$ when $z_1\in \bD\setminus E$.
Thus the same $R$ gives a smooth inverse branch $t=t(z_1,x_2)$
on $(\bD\setminus E)\times(-R,R)$ satisfying
$$
x_2=v_t(z_1,t(z_1,x_2)),
\qquad
|t(z_1,x_2)|
\le\kappa\operatorname{dist}(z_1,E).
$$
We extend $t$ to $\bD\times(-R,R)$ by setting
$t(\xi,x_2)=0$ for $\xi\in E$.
The preceding bound implies that $t(z_1,x_2)\to0$ as $z_1\to E$,
uniformly in $x_2$.
Together with the smoothness of $t$ on
$(\bD\setminus E)\times(-R,R)$, this proves that the extension
is continuous.

By the prescribed values of $W$ on $E$, the identity
$$
W(z_1,x_2)
=4\lp x_2t(z_1,x_2)-v(z_1,t(z_1,x_2))\rp
$$
holds throughout $\bD\times(-R,R)$.
Since $\bD$ and $v$ are bounded, we obtain
$$
\begin{aligned}
|W(z_1,x_2)|
&\le4R|t(z_1,x_2)|+4\|v\|_\infty\\
&\le\frac R2\operatorname{dist}(z_1,E)+4\|v\|_\infty
\le C.
\end{aligned}
$$
Since $v$ and $t$ are continuous, we conclude $W$ is continuous.

\subsection{Verification of the equation and the singular Laplacian}
We now verify the complex Monge--Amp\`ere equation. For $0<\eps\le1/4$, $z_1\in \bD$, and
$|t|\le\lp\operatorname{dist}(z_1,E)+\eps\rp/8$, 
$$\Delta_{z_1}v_\eps+(v_\eps)_{tt}
=-3\eps^2\int_E\int_{-L}^{L}
\frac{|s|^{1/4}\,ds\,d\nu(a)}
{\lp|z_1-a|^2+(t-s)^2+\eps^2\rp^{5/2}}
\le0.$$

Let $W_\eps$ be the normalized partial Legendre transform of $v_\eps$
on $\bD\times(-R,R)$, defined by
$$
W_\eps(z_1,x_2)
=4\lp x_2t_\eps-v_\eps(z_1,t_\eps)\rp,
\qquad t_\eps=t_\eps(z_1,x_2).
$$
Since $t_\eps\to t$ and $v_\eps\to v$ locally uniformly,
we obtain $W_\eps\to W$ locally uniformly on $\bD\times(-R,R)$. 
We regard $W_\eps(z_1,\text{Re}\,z_2)$ and $W(z_1,\text{Re}\,z_2)$
as functions on
$\bD\times\{z_2\in\bC\mid |\text{Re}\,z_2|<R\}$,
still denoted by $W_\eps$ and $W$.
In the formulas below, all derivatives of $v_\eps$ are evaluated at
$(z_1,t_\eps(z_1,x_2))$, where $x_2=\text{Re}\,z_2$.
Differentiating the partial Legendre transform gives
$$
\begin{aligned}
&(W_\eps)_{1\bar1}
=-\Delta_{z_1}v_\eps
+\frac{4|(v_\eps)_{z_1t}|^2}{(v_\eps)_{tt}},\\
&(W_\eps)_{1\bar2}
=-\frac{2(v_\eps)_{z_1t}}{(v_\eps)_{tt}},\qquad
(W_\eps)_{2\bar2}=\frac1{(v_\eps)_{tt}}.
\end{aligned}
$$
Taking the determinant, we obtain
\begin{equation*}
\det\bigl((W_\eps)_{j\bar k}\bigr)
=-\frac{\Delta_{z_1}v_\eps}{(v_\eps)_{tt}}
=1-\frac{\Delta_{z_1}v_\eps+(v_\eps)_{tt}}{(v_\eps)_{tt}}
\ge1.
\end{equation*}
Together with $(W_\eps)_{2\bar2}>0$, this proves that $W_\eps$
is strictly plurisubharmonic.

Let $\chi\in C_c^\infty(B_1(0))$ be nonnegative with
$\int_{\bC}\chi\,dV=1$, and set
$\chi_\eps(z_1)=\eps^{-2}\chi(z_1/\eps)$ and
$g_\eps=g*\chi_\eps$.
Then $g_\eps$ is smooth and subharmonic,
$g_\eps\to g$ locally uniformly, and
$$
\Delta_{z_1}g_\eps=2\pi(\nu*\chi_\eps).
$$
The functions
$$
u_{4,\eps}(z_1,z_2)
=W_\eps(z_1,\text{Re}\,z_2)+g_\eps(z_1)
$$
are smooth and strictly plurisubharmonic and converge locally
uniformly to $u_4$.
Since $g_\eps$ depends only on $z_1$,
$$
\det\bigl((u_{4,\eps})_{j\bar k}\bigr)
=1-\frac{\Delta_{z_1}v_\eps+(v_\eps)_{tt}}{(v_\eps)_{tt}}
+\frac{\Delta_{z_1}g_\eps}{4(v_\eps)_{tt}}.
$$

Now, we are going to show the right-hand side converges to $1$ in $L^1$ sense.
Fix $D'\Subset \bD$ and $0<r<R$.
Since
$$
\operatorname{supp}(\Delta_{z_1}g_\eps)
\subset\{z_1\in\bC\mid\operatorname{dist}(z_1,E)\le\eps\},
$$
\eqref{teps-bound} gives $|t_\eps(z_1,\pm r)|\le\eps/4$
on this support.
By the changing of variables, for every fixed $\eps>0$, there is $t=t_{\eps}(z_1,x_2)$. Hence $dt=(t_{\eps})_{x_2}dx_2=\frac{1}{(v_{\eps})_{tt}}dx_2$. We then obtain
$$
\begin{aligned}
\int_{D'\times(-r,r)}
\frac{\Delta_{z_1}g_\eps}{4(v_\eps)_{tt}}
\,dV(z_1)\,dx_2
&=\frac14\int_{D'}\lp t_\eps(z_1,r)-t_\eps(z_1,-r)\rp\Delta_{z_1}g_\eps
\,dV(z_1)\\
&\le\frac{\eps}{8}
\int_{D'}\Delta_{z_1}g_\eps\,dV(z_1)\le\frac{\pi}{4}\eps\to0.
\end{aligned}
$$

To estimate $-\frac{\Delta_{z_1}v_\eps+(v_\eps)_{tt}}{(v_\eps)_{tt}}$, we use \eqref{teps-bound} to obtain
$$
|z_1-a|^2+\lp t_\eps(z_1,x_2)-s\rp^2+\eps^2
\ge\frac12\lp|z_1-a|^2+s^2+\eps^2\rp,
\qquad a\in E.
$$
Using the uniform lower bound for $(v_\eps)_{tt}$,
integrating in $x_2$, we obtain
$$
\begin{aligned}
&0\le
\int_{\bD\times(-r,r)}
-\frac{\Delta_{z_1}v_\eps+(v_\eps)_{tt}}
       {(v_\eps)_{tt}}
\,dV(z_1)\,dx_2\\
&\qquad\le
C\eps^2\int_E\int_{-L}^L |s|^{1/4}
\left(
\int_{\bD}
\frac{dV(z_1)}
     {\lp|z_1-a|^2+s^2+\eps^2\rp^{5/2}}
\right)\,ds\,d\nu(a)\\
&\qquad\le
C\eps^2\int_E\int_{-L}^L |s|^{1/4}
\left(
2\pi\int_0^\infty
\frac{\tau\,d\tau}
     {\lp\tau^2+s^2+\eps^2\rp^{5/2}}
\right)\,ds\,d\nu(a)\\
&\qquad=
\frac{2\pi C}{3}\eps^2
\int_E\int_{-L}^L
\frac{|s|^{1/4}}{(s^2+\eps^2)^{3/2}}
\,ds\,d\nu(a)\\
&\qquad\le
C\eps^{-1}\int_0^\eps s^{1/4}\,ds
+C\eps^2\int_\eps^L s^{-11/4}\,ds\\
&\qquad=
\frac{4C}{5}\eps^{1/4}
+\frac{4C}{7}\eps^2\lp\eps^{-7/4}-L^{-7/4}\rp\le C\eps^{1/4}\longrightarrow0
\qquad\text{as }\eps\to0.
\end{aligned}
$$
Here we have used $\nu(E)=1$, and $C$ is independent of $\eps$.

Hence we choose $I\subset \bR$ to be a bounded interval, 
$$
\int_{D'\times(-r,r)\times I}
\left|\det\bigl((u_{4,\eps})_{j\bar k}\bigr)-1\right|
\,dV(z_1)\,dx_2\,dy_2
\le C\lp\eps^{1/4}+\eps\rp\to0.
$$
This proves
$$
\det\bigl((u_{4,\eps})_{j\bar k}\bigr)\to1
\qquad\text{in }L^1_{\mathrm{loc}}(\Om).
$$
By Bedford--Taylor continuity under local uniform convergence, we obtain
$$
\lp dd^cW\rp^2=\lp dd^cu_4\rp^2=\beta_2
\qquad\text{on }\Om.
$$

We next show that the distributional Laplacian of \(u_4\) has a singular part. Since $W(z_1,\text{Re}\,z_2)$ is plurisubharmonic and
$\Delta_{z_1}g=2\pi\nu$, we have
$$
\Delta u_4\ge2\pi\nu\otimes dV(z_2)
$$
in the sense of distributions on $\Om$.
The right-hand side is a nonzero positive measure supported on
$E\times\{z_2\in\bC\,\big|\,|\text{Re}\,z_2|<R\}$,
which has zero four-dimensional Lebesgue measure.
Thus $\Delta u_4$ has a nonzero singular part and cannot be
represented by a function in $L^1_{\mathrm{loc}}(\Om)$.
Finally, a standard rescaling of the form $c_1u_4(z_1,c_2z_2)$
gives the solution on $\bD^2$ claimed in the theorem.

\begin{rem}
The construction also applies to other compact subsets of \(\mathbb C\) that support measures satisfying the same two-sided mass bounds with \(d\in(0,1)\), provided that the weight \(|s|^{1/4}\) is replaced by \(|s|^\alpha\) for some \(0<\alpha<1-d\). A conversation with a colleague initially led us to
consider fractal sets arising in complex dynamics; we later
chose the Cantor set.

After completing the construction, we noticed a connection with
Mooney's work. Mooney \cite{Mooney16} constructed a convex solution
of the degenerate real Monge--Amp\`ere equation
$0\le\det D^2u\le1$ whose distributional Hessian has a Cantor part.
His construction uses self-similar superpositions of translated
and anisotropically rescaled convex functions. Our construction
combines the logarithmic potential of a Cantor measure with
a partial Legendre transform of a weighted Newton potential,
producing a singular part in the Laplacian while preserving
the constant complex Monge--Amp\`ere density $1$.
\end{rem}

\section{Further examples}\label{sec:further-examples}
In this section, we construct further explicit solutions of geometric and analytic interest using partial Legendre transforms and holomorphic pullbacks.

\subsection{Solutions from Warren’s harmonic dual}
The harmonic dual \(h_W\) of Warren’s solution also gives the following families through reflection, translation, and addition of a harmonic quadratic polynomial.

\subsubsection{Interpolation between Warren’s harmonic dual and its reflection}

We first combine Warren’s harmonic dual with its reflection in \(t\). For \(0\leq\alpha\leq2\), set
\[
v_\alpha(z_1,t)
=\left(1-\frac{\alpha}{2}\right)h_W(z_1,t)
+\frac{\alpha}{2}h_W(z_1,-t).
\]
This function is harmonic on \(\mathbb C^*\times\mathbb R\), with
\[
(v_\alpha)_{tt}
=\frac{1}{2\sqrt{|z_1|^2+t^2}}>0.
\]
The partial Legendre relation gives
\[
x_2=(v_\alpha)_t
=\frac12\log\left(t+\sqrt{|z_1|^2+t^2}\right)-\frac{\alpha}{2}\log |z_1|.
\]
Consequently, the normalized partial Legendre transform yields the explicit family of solutions
\[
U_\alpha(z_1,z_2)
=|z_1|^\alpha e^{2x_2}
+|z_1|^{2-\alpha}e^{-2x_2}.
\]
At \(\alpha=0\), we recover Warren’s solution \(U_W\), while at \(\alpha=2\), we obtain its reflection \(U_W(z_1,-z_2)\). For \(0<\alpha<2\), \(U_\alpha\) extends continuously and plurisubharmonically across \(\{z_1=0\}\). The pluripolar extension result above therefore shows that every member of this family satisfies
\[
\det\bigl((U_\alpha)_{i\bar j}\bigr)=1
\]
on \(\mathbb C^2\) in the Bedford–Taylor sense. For \(0<\alpha<2\), the optimal local Hölder exponent of \(U_\alpha\) is \(\min\{\alpha,2-\alpha\}\), which varies continuously with \(\alpha\) and ranges over \((0,1]\).

\subsubsection{Superpositions of translates of Warren’s harmonic dual}
Let \(\xi\in\mathbb C\) and \(s\in\mathbb R\), and set
\[
v(z_1,t)=h_W(z_1,t)+h_W(z_1-\xi,t-s).
\]
This function is harmonic and strictly convex in \(t\) on
\((\mathbb C\setminus\{0,\xi\})\times\mathbb R\). 
Write
\[
r_0=\sqrt{|z_1|^2+t^2},
\qquad
r_1=\sqrt{|z_1-\xi|^2+(t-s)^2}.
\]
The partial Legendre relation gives
\[
e^{2x_2}=(t+r_0)(t-s+r_1),
\]
and hence
\[
4v^*=4(x_2t-v)
=2(r_0+r_1)+2s\log(t-s+r_1).
\]
Eliminating \(t\), we obtain
\[
\begin{aligned}
U(z_1,z_2)
={}&2\sqrt{s^2+e^{-2x_2}
\bigl(e^{2x_2}+|z_1|^2\bigr)
\bigl(e^{2x_2}+|z_1-\xi|^2\bigr)}\\
&+2s\log
\frac{e^{2x_2}+|z_1-\xi|^2}
{s+\sqrt{s^2+e^{-2x_2}
\bigl(e^{2x_2}+|z_1|^2\bigr)
\bigl(e^{2x_2}+|z_1-\xi|^2\bigr)}}.
\end{aligned}
\]
The function \(U\) extends smoothly to all of \(\mathbb C^2\) and is a strictly plurisubharmonic entire solution of \(\det(U_{i\bar j})=1\). When \(s=0\), the formula reduces to
\[
U(z_1,z_2)
=2e^{-x_2}
\sqrt{\bigl(e^{2x_2}+|z_1|^2\bigr)
\bigl(e^{2x_2}+|z_1-\xi|^2\bigr)}.
\]

\subsubsection{Addition of a harmonic quadratic polynomial}
For \(m\geq0\), set
\[
v_m(z_1,t)=h_W(z_1,t)+m\left(t^2-\frac12|z_1|^2\right).
\]
This function is harmonic and strictly convex in \(t\) on \(\mathbb C^*\times\mathbb R\). We show that applying the normalized partial Legendre transform to \(v_m\), followed by pullback under \(\Pi\), yields LeBrun’s Taub–NUT potential. This is the same procedure used in Section \ref{sec:proper-w21} to construct a continuous exhaustion solution that fails to belong to \(W^{2,1}_{\mathrm{loc}}\).

The Taub–NUT metric originates in the Lorentzian solutions of Taub \cite{Taub51} and Newman, Tamburino and Unti \cite{NTU63}. Its self-dual Euclidean version, studied by Hawking \cite{Hawking77} as a gravitational instanton, is a complete, nonflat hyperkähler metric on \(\mathbb R^4\). LeBrun \cite[p.302]{LeB} gave the following global Kähler potential for the Taub--NUT metric on \(\mathbb C^2\):
\[
\Phi_m(\zeta_1,\zeta_2)
=p^2+q^2+m\left(p^4+q^4\right),
\]
where \(p,q\geq0\) are determined by
\[
|\zeta_1|=p\,e^{m(p^2-q^2)},
\qquad
|\zeta_2|=q\,e^{m(q^2-p^2)}.
\]
\(\Phi_m\) is a smooth, strictly plurisubharmonic solution of \(\det((\Phi_m)_{i\bar j})=1\) on \(\mathbb C^2\). 

Let \(U_m(z_1,z_2)=4v_m^*(z_1,\operatorname{Re}z_2)\) be the normalized partial Legendre transform of \(v_m\). We prove that, on \((\mathbb C^*)^2\),
\[
(U_m\circ\Pi)(\zeta_1,\zeta_2)
=4v_m^*\bigl(\zeta_1\zeta_2,\log|\zeta_2|\bigr)
=\Phi_m(\zeta_1,\zeta_2).
\]

The partial Legendre relation gives
\[
x_2=(v_m)_t
=\frac12\log\left(t+\sqrt{|z_1|^2+t^2}\right)+2mt.
\]
Hence
\[
4v_m^*(z_1,x_2)=2\sqrt{|z_1|^2+t^2}+4mt^2+2m|z_1|^2.
\]

Now take \((\zeta_1,\zeta_2)\in(\mathbb C^*)^2\), and let \(p,q\) be the corresponding parameters in LeBrun’s formula. Then \(p,q>0\). Set \(z_1=\zeta_1\zeta_2\) and \(x_2=\log|\zeta_2|\). The defining relations for \(p\) and \(q\) give
\[
|z_1|=pq,\qquad x_2=\log q+m(q^2-p^2).
\]
Comparing this expression for \(x_2\) with the partial Legendre relation, we observe that
\[
t=\frac{q^2-p^2}{2}
\]
satisfies that relation. By strict convexity, this is the unique value of \(t\). Consequently,
\[
\begin{aligned}
(U_m\circ\Pi)(\zeta_1,\zeta_2)
&=p^2+q^2+m(q^2-p^2)^2+2mp^2q^2\\
&=p^2+q^2+m\left(p^4+q^4\right)\\
&=\Phi_m(\zeta_1,\zeta_2).
\end{aligned}
\]
The above construction can be summarized by the following sequence of transformations:
\[
\rho
\xrightarrow{\;\Psi^*\;}
U_W
\xrightarrow{\;\text{inverse normalized PL}\;}
h_W
\xrightarrow{\;+\;mQ\;}
v_m
\xrightarrow{\;\text{normalized PL}\;}
U_m
\xrightarrow{\;\Pi^*\;}
\Phi_m.
\]
Here PL stands for partial Legendre transformation, and
\[
\rho(\zeta_1,\zeta_2)=|\zeta_1|^2+|\zeta_2|^2,
\qquad
Q(z_1,t)=t^2-\frac12|z_1|^2.
\]

\begin{rem}
While completing this paper, we became aware of the classical work of Hitchin, Karlhede, Lindstr\"om and Ro\v{c}ek \cite[§2(A)]{HKLR87}, where hyperk\"ahler metrics are constructed by applying partial Legendre transformations, with different sign and normalization conventions, to superpositions of translates of the function
\[
-2h_W(2z_1,t)+t\log(2|z_1|).
\]
They note that the logarithmic term can be dropped, since it only induces a local holomorphic coordinate change after the transform. They also point out that adding a harmonic quadratic polynomial to the basic function yields Taub–NUT metrics after the Legendre transformation.

Here we directly use Warren’s harmonic dual \(h_W\) and explicitly compute the normalized partial Legendre transform of the sum of two translates. We also give an explicit holomorphic change of variables identifying the transform of \(h_W+mQ\) with LeBrun’s Taub–NUT potential.
\end{rem}

\subsection{Dual of a cubic harmonic polynomial}
Consider the cubic harmonic polynomial
\[
v(z_1,t)=\frac23t^3-t|z_1|^2.
\]
On \(\{t>0\}\), we have \(v_{tt}=4t>0\), and the partial Legendre relation is
\[
x_2=v_t=2t^2-|z_1|^2.
\]
Its normalized partial Legendre transform is therefore
\[
u(z_1,z_2)
=4v^*(z_1,\re z_2)
=\frac{4\sqrt2}{3}\left(|z_1|^2+\re z_2\right)^{3/2}.
\]
This is a smooth, strictly plurisubharmonic solution of
\[
\det\bigl(u_{i\bar j}\bigr)=1
\]
on \(\{(z_1,z_2)\in\mathbb C^2\,|\,|z_1|^2+\re z_2>0\}\). 

The corresponding two-dimensional real solution, after normalization, is the function
\[
w(x,y)=\frac23(x^2+y)^{3/2},
\]
defined on \(\Omega=\{(x,y)\in\mathbb R^2\,|\,x^2+y>0\}\). Its Hessian is positive definite on \(\Omega\), and it satisfies
\[
\det D^2w=1.
\]

This function gives an explicit counterexample to global \(C^2\) regularity for \(C^{2,1}\) boundary data. Indeed, consider the disk
\[
D=\left\{(x,y)\in\mathbb R^2\,\middle|\,x^2+(y-1)^2<1\right\}.
\]
Note that \(w\) is smooth and strictly convex in \(D\). Near the origin, the boundary is given by
\[
y(x)=1-\sqrt{1-x^2},
\]
and
\[
w(x,y(x))
=\frac23|x|^3
\left(1+\frac1{1+\sqrt{1-x^2}}\right)^{3/2}.
\]
Consequently,
\[
w|_{\partial D}\in C^{2,1}(\partial D)\setminus C^3(\partial D).
\]
However,
\[
w_{yy}(0,y)=\frac1{2\sqrt{y}}\longrightarrow+\infty
\qquad\text{as }y\to0^+.
\]
This shows that the $C^3$ assumption on the boundary data in the global
regularity theorems \cite[Theorem~1]{W96} and \cite[Theorem~1.1]{TW08}
cannot be weakened to $C^{2,1}$, even on a disk with constant right-hand
side. Wang \cite[Example~1]{W96} previously established this failure of
boundary regularity through a Dirichlet problem; here the solution is
given explicitly.

\subsection{A new family of entire toric solutions}
Let \(n\geq2\), and let \(A=(a_{ij})\) be an invertible $n\times n$ row-stochastic matrix, that is,
\[
a_{ij}\geq0,
\qquad
\sum_{j=1}^n a_{ij}=1
\quad\text{for }i=1,\ldots,n.
\]
Define
\[
u(z_1,\ldots,z_n)=
|\det A|^{-\frac{2}{n}}
\sum_{j=1}^n\prod_{i=1}^n|z_i|^{2a_{ij}}.
\]
Thus \(u\) is continuous on \(\mathbb C^n\) and toric. We show that it solves
\[
\det\bigl(u_{i\bar k}\bigr)=1
\]
on \(\mathbb C^n\) in both the Bedford–Taylor and viscosity senses. The verification is conveniently carried out using a holomorphic pullback interpretation of \(u\).

For any \(p\in(\mathbb C\setminus\{0\})^n\), choose a sufficiently small neighborhood \(U\) of \(p\) and local branches of \(\log z_i\) on \(U\). Define the holomorphic map
\[
F:U\longrightarrow\mathbb C^n,
\qquad
F_j(z)=|\det A|^{-\frac{1}{n}}\prod_{i=1}^nz_i^{a_{ij}},
\quad j=1,\ldots,n,
\]
where the powers are taken with respect to these branches. Let \(\rho(w)=|w|^2\). Since \(u|_U=\rho\circ F\) and \(F\) is holomorphic,
\[
dd^cu=F^*(dd^c\rho).
\]
The row-sum condition gives
\[
\det D_{\mathbb C}F
=
\frac{\det A}{|\det A|}
\prod_{i=1}^nz_i^{\sum_{j=1}^na_{ij}-1}
=
\frac{\det A}{|\det A|}.
\]
Thus \(F\) is locally biholomorphic and preserves the standard volume form. Consequently, \(u\) is strictly plurisubharmonic on \(U\), and
\[
(dd^cu)^n=F^*\bigl((dd^c\rho)^n\bigr)=F^*\beta_n=\beta_n.
\]
Since \(p\) is arbitrary, the equation holds on \((\mathbb C\setminus\{0\})^n\).

Set
\(
Z=\{z\in\mathbb C^n\,|\,z_1\cdots z_n=0\}.
\)
Since \(u\) is continuous and locally bounded near \(Z\), which is closed and locally complete pluripolar, Lemma \ref{lem:pluripolar-extension} gives
\[
u\in PSH(\mathbb C^n)\cap C(\mathbb C^n),
\qquad
(dd^c u)^n=\beta_n
\quad\text{on }\mathbb C^n.
\]
By Proposition \ref{prop:bt-viscosity}, \(u\) is also a viscosity solution.

In complex dimension two, write
\[
A=
\begin{pmatrix}
a &\,1-a\\
b &\,1-b
\end{pmatrix},
\qquad
a,b\in[0,1],\quad a\ne b.
\]
The resulting family is
\[
u_{a,b}(z_1,z_2)
=
\frac{
|z_1|^{2a}|z_2|^{2b}
+
|z_1|^{2-2a}|z_2|^{2-2b}
}{|a-b|}.
\]

This recovers the two families of entire toric solutions in complex dimension two from our earlier work \cite{WangWang25}. For \(0\leq\beta<1\), taking \(a=\beta/2\) and \(b=1/2\) gives the first family
\[
v_\beta(z_1,z_2)
=
\frac{2}{1-\beta}|z_2|
\left(|z_1|^\beta+|z_1|^{2-\beta}\right).
\]
For \(0\leq\gamma<2\), we take \(a=\gamma/2\) and \(b=1\) and adjust the coefficients to obtain the second family
\[
w_\gamma(z_1,z_2)
=
|z_1|^\gamma|z_2|^2
+
\frac{4}{(2-\gamma)^2}|z_1|^{2-\gamma}.
\]

\end{document}